\documentclass[reqno]{amsart}
\usepackage{graphicx}
\usepackage{amsmath}
\usepackage{amsfonts}
\usepackage{amssymb}
\usepackage{hyperref}
\providecommand{\U}[1]{\protect\rule{.1in}{.1in}}

\newtheorem{theorem}{Theorem}

\newtheorem{corollary}[theorem]{Corollary}

\newtheorem{example}[theorem]{Example}

\newtheorem{lemma}[theorem]{Lemma}

\newtheorem{proposition}[theorem]{Proposition}
\newtheorem{remark}[theorem]{Remark}

\numberwithin{theorem}{section}
\numberwithin{equation}{section}
\begin{document}
\title{Spectral results for free random variables}
\author{Brian C. Hall}
\address{Department of Mathematics, University of Notre Dame, Notre Dame, IN 46556, USA}
\email{bhall@nd.edu}
\author{Ching-Wei Ho}
\address{Institute of Mathematics, Academia Sinica, Taipei 10617, Taiwan}
\email{chwho@gate.sinica.edu.tw}

\subjclass{46L54, 60B20}
\keywords{Free probability, Brown measure, spectrum, log potential, free Brownian motions, random matrices, matrix Brownian motion}

\begin{abstract}
Let $(\mathcal{A},\mathrm{tr})$ be a von Neumann algebra with a faithful,
normal trace $\mathrm{tr}:\mathcal{A}\rightarrow\mathbb{C}.$ For each
$a\in\mathcal{A},$ define
\[
S(\lambda,\varepsilon)=\mathrm{tr}[\log((a-\lambda)^{\ast}(a-\lambda
)+\varepsilon)],\quad\lambda\in\mathbb{C},~\varepsilon>0,
\]
so that the limit as $\varepsilon\rightarrow0^{+}$ of $S$ is the log potential
of the Brown measure of $a.$ Suppose that for a fixed $\lambda\in\mathbb{C},$
the function%
\[
\varepsilon\mapsto\frac{\partial S}{\partial\varepsilon}(\lambda
,\varepsilon)=\mathrm{tr}[\log((a-\lambda)^{\ast}(a-\lambda)+\varepsilon
)^{-1}]
\]
admits a real analytic extension to a neighborhood of $0$ in $\mathbb{R}.$
Then we will show that $\lambda$ is outside the spectrum of $a.$

We will apply this result to several examples involving circular and elliptic
elements, as well as free multiplicative Brownian motions. In most cases, we
will show that the spectrum of the relevant element $a$ coincides with the
support of its Brown measure.

\end{abstract}
\maketitle
\tableofcontents

{\centering
\textit{Dedicated to Leonard Gross on the occasion of his 95th birthday}

}

\section{Introduction}

In this paper, we introduce a new characterization (Theorem \ref{main.thm}) of
the spectrum of an element in a tracial von Neumann algebra. We then apply
this result to several examples, such as (1) the sum of an arbitrary element
and a freely independent circular element, (2) more generally, the sum of an
arbitrary element and a freely independent elliptic element, and (3) the
product of a unitary element and a freely independent free multiplicative
Brownian motion. Under suitable assumptions, we establish equality of the
spectrum and the support of the Brown measure for these examples.

We now describe the origins of this line of research in the work of Leonard
Gross. Let $K$ be a connected compact Lie group, where the unitary group
$K=U(N)$ will be a key example. The paper \cite{Gross} of Gross proved
ergodicity for the action of the finite-energy loop group over $K$ on the
continuous loop group with the pinned Wiener measure. A by-product of Gross's
proof was a Fock-space or \textquotedblleft chaos\textquotedblright%
\ decomposition of the $L^{2}$ space over $K$ with respect to a heat kernel
measure $\rho_{t}$. This result then motivated the introduction by Hall
\cite{Hall1} of the Segal--Bargmann transform for $K.$ The transform is a
unitary map from $L^{2}(K,\rho_{t})$ onto a holomorphic $L^{2}$ space of
functions on the complexification $K_{\mathbb{C}}$ of $K.$ In the case
$K=U(N),$ we have $K_{\mathbb{C}}=GL(N;\mathbb{C}),$ the group of all $N\times
N$ invertible matrices over $\mathbb{C}.$

A paper of Gross and Malliavin \cite{GrossMalliavin} then gave a stochastic
construction of the Segal--Bargmann over $K,$ using the Brownian motions in
$K$ and $K_{\mathbb{C}}$ and methods from \cite{Gross}. Finally, Biane
\cite{BianeJFA} essentially took the construction of Gross and Malliavin for
the case $K=U(N)$ and $K_{\mathbb{C}}=GL(N;\mathbb{C})$ and took the limit as
$N\rightarrow\infty.$ Biane's work indicated a close relationship between the
free unitary Brownian motion (large-$N$ limit of Brownian motion in $U(N)$)
and the free multiplicative Brownian motion (large-$N$ limit of Brownian
motion in $GL(N;\mathbb{C})$).

Biane's work then provided the motivation and technical tools for work of
Hall--Kemp \cite{HK} computing the support of the Brown measure of the free
multiplicative Brownian motion and then work of Driver--Hall--Kemp \cite{DHK}
computing the Brown measure itself. (Here \textquotedblleft Brown
measure\textquotedblright\ \cite{Brown} is a von Neumann algebra construction
that mimics the notion of eigenvalue distribution in random matrix theory.)
The paper \cite{DHK} introduced a new PDE\ method for computing Brown
measures, which has then been used in subsequent works of Ho--Zhong \cite{HZ},
Hall--Ho \cite{HallHo1,HallHo2}, Demni--Hamdi \cite{DemniHamdi}, and
Eaknipitsari--Hall \cite{EH}. We also mention the work of Zhong \cite{Zhong},
which does not use the PDE\ method but obtains similar formulas using free
probability methods. The present paper also uses the PDE\ method in the
applications of our general result.

We now discuss the results of the current paper. The Brown measure is defined
for an element in a tracial von Neumann algebra $\mathcal{A},$ that is, a von
Neumann algebra with a faithful, normal trace. (See Section \ref{Brown.sec}
for details.) In general, the (closed) support of the Brown measure of $a$ is
\textit{contained in} the spectrum of $a.$ In many examples, the support of
the Brown measure \textit{equals} the spectrum and it is desirable to obtain
conditions that would guarantee this equality. In the present paper, we
introduce a new characterization (Theorem \ref{main.thm}) of the spectrum of
an element of a tracial von Neumann algebra, which is well suited for use with
the PDE\ method of \cite{DHK}. We then apply this result to get conditions on
the spectrum in various examples.

We now briefly summarize the applications we will make of Theorem
\ref{main.thm}. We consider a element of the form~$x+c$, where $c$ is circular
and $x$ is freely independent of $x,$ or more generally $x+g,$ where $g$ is
elliptic and freely independent of $x.$ Assume that the spectrum of $x$
coincides with the support of its Brown measure, which will happen, for
example, if $x$ is normal. Then the spectrum of $x+c$ coincides with the
support of its Brown measure (Corollary \ref{equalSupports.cor}), and the same
result holds more generally for $x+g$ (Theorem \ref{ellipticEquality.thm}).
Meanwhile, consider the general free multiplicative Brownian motion
$b_{s,\tau}$ introduced in \cite[Section 2.1]{HallHo2} and let $u$ be a
unitary element that is freely independent of $b_{s,\tau}.$ Then the spectrum
of $ub_{s,\tau}$ coincides with the support of its Brown measure (Theorem
\ref{multEquality.thm}). In particular, we determine the spectrum of the free
multiplicative Brownian motion $b_{t},$ in the original form introduced by
Biane \cite[Section 4.2.1]{BianeJFA}. The spectrum of $b_{t}$ is equal to the
support of its Brown measure, which is the closure of the domain $\Sigma_{t}$
in \cite{DHK}. See Figure \ref{t4.fig}.%

\begin{figure}[ptb]
\centering
\includegraphics[scale=0.6]{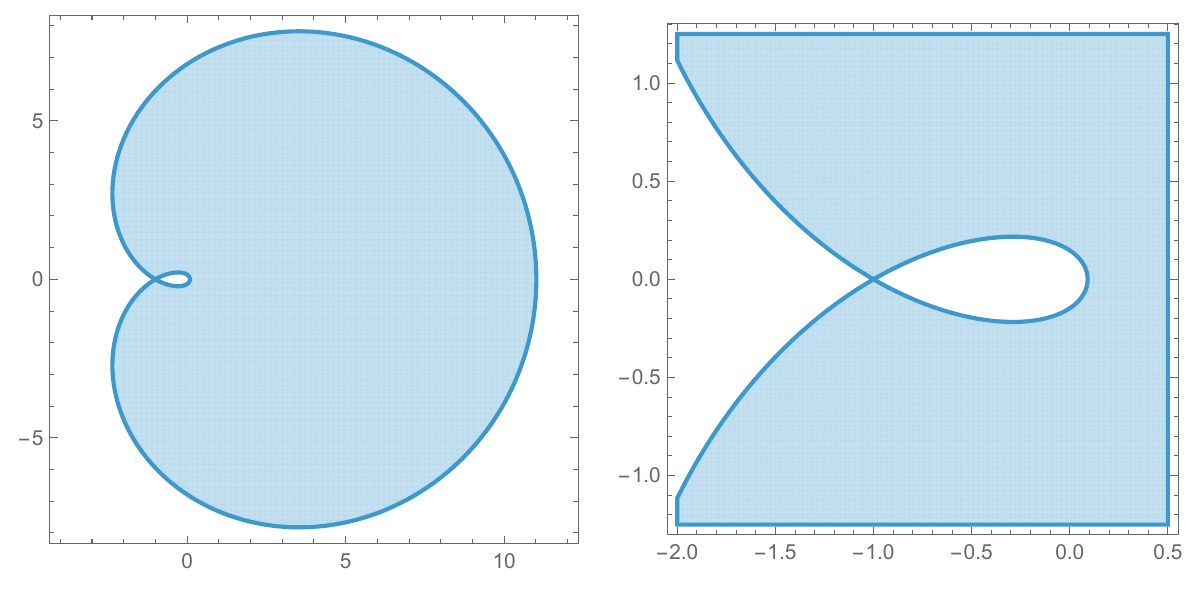}
\caption{The region $\Sigma_{t}$ from \cite{DHK} for $t=4$ (left), and a
detail thereof (right).}
\label{t4.fig}
\end{figure}

\section{A general result characterizing the spectrum}

\subsection{Brown measure\label{Brown.sec}}

Let $(\mathcal{A},\mathrm{tr})$ be a tracial von Neumann algebra, that is a
von Neumann algebra $\mathcal{A}$ together with a faithful, normal, tracial
state $\mathrm{tr}:\mathcal{A}\rightarrow\mathbb{C}.$ Here \textquotedblleft
faithful\textquotedblright\ means that $\mathrm{tr}[a^{\ast}a]>0$ for all
nonzero $a\in\mathcal{A},$ \textquotedblleft normal\textquotedblright\ means
that $\mathrm{tr}$ is continuous with respect to the weak operator topology,
and \textquotedblleft tracial\textquotedblright\ means that $\mathrm{tr}%
[ab]=\mathrm{tr}[ba]$ for all $a,b\in\mathcal{A}.$ For $a\in\mathcal{A},$ we
let $\left\vert a\right\vert $ be the non-negative square root of $a^{\ast}a.$

If $a\in\mathcal{A}$ is a normal operator, we can define the \textbf{law} (or
spectral distribution) $\mu_{a}$\ of $a$ using the spectral theorem as%
\begin{equation}
\mu_{a}(E)=\mathrm{tr}[\nu_{a}(E)] \label{lawDef}%
\end{equation}
for each Borel set $E,$ where $\nu_{a}$ is the projection-valued measure
associated to $a$ by the spectral theorem (e.g., \cite[Theorem 7.12]{QM}). The
measure $\mu_{a}$ is the unique compactly supported probability measure on
$\mathbb{C}$ satisfying%
\begin{equation}
\int_{\mathbb{C}}\lambda^{j}\bar{\lambda}^{k}~d\mu_{a}(\lambda)=\mathrm{tr}%
[a^{j}(a^{\ast})^{k}] \label{LawMoments}%
\end{equation}
for all non-negative integers $j$ and $k.$

Brown \cite{Brown}, extended the notion of law to elements that are not
necessarily normal, as follows. For $\lambda\in\mathbb{C},$ we define%
\begin{equation}
s(\lambda)=\mathrm{tr}[\log(\left\vert a-\lambda\right\vert ^{2})],
\label{sDef}%
\end{equation}
which may be computed in terms of the law of $\left\vert a-\lambda\right\vert
^{2}$ as%
\begin{equation}
s(\lambda)=\int_{0}^{\infty}\log(x)~d\mu_{\left\vert a-\lambda\right\vert
^{2}}(x). \label{sForm}%
\end{equation}
The value of $s$ is defined to be $-\infty$ if $\mu_{\left\vert a-\lambda
\right\vert ^{2}}$ has positive mass at 0; the value of $s$ may also be
$-\infty$ even if $\mu_{\left\vert a-\lambda\right\vert ^{2}}$ has no mass at 0.

Brown showed that $s(\lambda)$ is finite for Lebesgue-almost-every value of
$\lambda$ and is a subharmonic function of $\lambda.$ He then defined the
\textbf{Brown measure} $\mathrm{Br}_{a}$ as%
\begin{equation}
\mathrm{Br}_{a}=\frac{1}{4\pi}\Delta s, \label{smallSdef}%
\end{equation}
where $\Delta$ is the distributional Laplacian.

\begin{proposition}
\label{BrownProperties.prop}Properties of $\mathrm{Br}_{a}$ include:
\end{proposition}

\begin{enumerate}
\item \label{brownSpec.point}$\mathrm{Br}_{a}$ is a probability measure
supported on the spectrum of $a.$

\item The function $s$ is the log potential of $\mathrm{Br}_{a},$ that is, the
convolution of $\mathrm{Br}_{a}$ with the function $\log(\left\vert
z\right\vert ^{2}).$

\item $\mathrm{Br}_{a}$ agrees with $\mu_{a}$ if $a$ is normal.

\item We have%
\[
\int_{\mathbb{C}}\lambda^{j}~d\mathrm{Br}_{a}(\lambda)=\mathrm{tr}[a^{j}]
\]
for all non-negative integers $j.$

\item If $\mathcal{A}$ is the space of $N\times N$ matrices with complex
entries and $\mathrm{tr}$ is the normalized matrix trace, $\mathrm{tr}%
[a]=\frac{1}{N}\sum_{j=1}^{N}a_{jj},$ then $\mathrm{Br}_{a}$ is the
\textbf{empirical eigenvalue distribution} of $a,$ namely%
\[
\mathrm{Br}_{a}=\frac{1}{N}\sum_{j=1}^{N}\delta_{\lambda_{j}},
\]
where $\{\lambda_{1},\ldots,\lambda_{N}\}$ are the eigenvalues of $A.$
\end{enumerate}

We emphasize, however, that the Brown measure does not, in general, satisfy
(\ref{LawMoments}).

\subsection{The spectrum and the support of the Brown measure}

For any probability measure $\mu$ on $\mathbb{C},$ the \textbf{support} of
$\mu,$ denoted $\mathrm{supp}(\mu),$ is the smallest closed set of full
measure. In the case of the Brown measure of an element $a,$ we refer to
$\mathrm{supp}(\mathrm{Br}_{a})$ as the \textbf{Brown support} of $a.$ In
light of Point \ref{brownSpec.point} of Proposition \ref{BrownProperties.prop}%
, the Brown support of any $a$ is \textit{contained in} the spectrum of $a.$
Although, in many cases, the Brown support and the spectrum are actually
equal, this is not always the case. Thus, it is desirable to identify tools
that can allow us to prove equality of the Brown support and the spectrum in
certain cases.

The following example shows that the spectrum and Brown support can differ for
$R$-diagonal elements, that is, elements having the form of a Haar unitary
times a freely independent non-negative element.

\begin{example}
[Haagerup--Larsen]\label{hl.ex}Suppose $h$ is a non-negative self-adjoint
element such that the spectrum of $h$ contains 0 but $h$ has an $L^{2}$
inverse, meaning that
\[
\mathrm{tr}\left[  h^{-2}\right]  :=\int_{0}^{\infty}\frac{1}{\xi^{2}}%
~d\mu_{h}(\xi)<\infty.
\]
Let $x=uh,$ where $u$ is a Haar unitary that is freely independent of $h.$
Then the spectrum of $x$ is the disk%
\[
\sigma(x)=\left\{  \left.  \lambda\in\mathbb{C}\right\vert \left\vert
\lambda\right\vert ^{2}\leq\mathrm{tr}[h^{2}]\right\}  ,
\]
but the support of the Brown measure of $x$ is the annulus%
\[
\mathrm{supp}(\mathrm{Br}_{x})=\left\{  \lambda\in\mathbb{C}\left\vert
\frac{1}{\mathrm{tr}[h^{-2}]}\leq\left\vert \lambda\right\vert ^{2}%
\leq\mathrm{tr}[h^{2}]\right.  \right\}  .
\]

\end{example}

The preceding result is a combination of Theorem 4.4(i) and Proposition 4.6 in
\cite{HaagerupLarsen}.

The following elementary result says that for normal operators the spectrum
and the Brown support agree.

\begin{proposition}
\label{normalSupport.prop}If $a\in\mathcal{A}$ is normal, the closed support
of $\mathrm{Br}_{a}=\mu_{a}$ is equal to the spectrum of $a.$
\end{proposition}

\begin{proof}
If $P$ is a nonzero self-adjoint projection, then $\mathrm{tr}[P]=\mathrm{tr}%
[P^{2}]>0,$ by the faithfulness of the trace. It follows that the law $\mu
_{a}$ of $a$ has the same sets of measure zero as the projection-valued
measure $\nu_{a}$ associated to $a$ by the spectral theorem. Thus, the closed
support of $\mu_{a}$ is the same as the closed support of $\nu_{a}.$ Part of
the spectral theorem states that $\nu_{a}$ is supported on the spectrum of $a$
(which is a closed set), showing that the closed support of $\nu_{a}$ is
contained in the spectrum.

We now show that the closed support of $\nu_{a}$ contains the spectrum of $a.$
If not, there would be a point $\lambda$ that is in $\sigma(a)$ but outside
the closed support of $\nu_{a},$ which means that $\nu_{a}(D_{\lambda
}(\varepsilon))=0$ for some open disk centered at $\lambda$ with radius
$\varepsilon.$ Then consider the bounded function $f$ given by
\[
f(x)=\left\{
\begin{array}
[c]{cc}%
0 & x\in D_{\lambda}(\varepsilon)\\
\frac{1}{x-\lambda} & x\notin D_{\lambda}(\varepsilon)
\end{array}
\right.  .
\]
Then $f(a)$ is a bounded normal operator. Meanwhile, let $g(x)=x-\lambda,$ so
that $g(a)=a-\lambda I.$ Then by the multiplicativity of the functional
calculus for bounded measurable functions (e.g. Theorem 7.7 in \cite{QM}), we
have%
\[
f(a)g(a)=g(a)f(a)=(fg)(a).
\]
But the function $fg$ equals $1$ for $\nu_{a}$-almost-every $x,$ so
$(fg)(a)=I.$ This shows that $f(a)$ is a bounded inverse of $g(a)=a-\lambda.$
\end{proof}

\subsection{The regularized log potential and its derivative}

It is convenient to introduce the \textbf{regularized log potential} $S$ of
$a\in\mathcal{A}$ as%
\begin{equation}
S(\lambda,\varepsilon)=\mathrm{tr}[\log(\left\vert a-\lambda\right\vert
^{2}+\varepsilon)],\quad\lambda\in\mathbb{C},~\varepsilon>0. \label{Sdef}%
\end{equation}
(See Section 11.5 of the monograph \cite{MingoSpeicher} of Mingo and
Speicher.) Then $S$ is a $C^{\infty}$ function of both $\lambda$ and
$\varepsilon$ and is subharmonic \cite[Equation (11.8)]{MingoSpeicher} as a
function of $\lambda$ for fixed $\varepsilon.$ For $\lambda\in\mathbb{C}$
fixed, $S(\lambda,\varepsilon)$ decreases as $\varepsilon$ decreases, so that
the limit as $\varepsilon\rightarrow0^{+}$ exists, possibly equal to
$-\infty.$ After separating the log function into its positive and negative
parts and applying monotone convergence, we find that%
\[
s(\lambda)=\lim_{\varepsilon\rightarrow0^{+}}S(\lambda,\varepsilon)
\]
for all $\lambda\in\mathbb{C}.$

\begin{remark}
Sometimes, a different convention is used, in which $\varepsilon$ is replaced
by $\varepsilon^{2}$ on the right-hand side of (\ref{sDef}). In the main
results below, it is extremely important to distinguish between the
\textquotedblleft$\varepsilon$ regularization\textquotedblright\ and the
\textquotedblleft$\varepsilon^{2}$ regularization.\textquotedblright\ See
Remark \ref{epsilonSquared.rem}.
\end{remark}

The function $S$ is a regularization of the log potential $s$ of
$\mathrm{Br}_{a},$ in the sense that $S$ is a smooth function that
approximates $s$ for small $\varepsilon.$ It is important to note, however,
that $S$ cannot be computed from $s$; to compute $S,$ one needs information
about the element $a$ that cannot (in general) be computed just from the
function $s.$ (Thus, for example, $S$ cannot be computed as the convolution of
$s$ with some mollifier function.) In particular, $S$ is not determined by
$\mathrm{Br}_{a}$; if it were, it would also be determined by $s$, which is
the log potential of $\mathrm{Br}_{a}$.

Although the function $S$ was introduced as a convenient regularization of the
log potential $s$ of $\mathrm{Br}_{a},$ it plays a more fundamental role in
certain Brown measure calculations. Specifically, Driver--Hall--Kemp
\cite{DHK} consider the log potential $S(t,\lambda,\varepsilon)$ of Biane's
free multiplicative Brownian motion $b_{t}.$ Then \cite{DHK} shows that $S$
satisfies a PDE in which $\varepsilon$ appears as one of the variables. (See
Section \ref{ubt.sec}.) One cannot simply set $\varepsilon=0$ in the PDE
because derivatives with respect to $\varepsilon$ appear. Further works using
a PDE\ for the regularized log potential include those of Ho--Zhong \cite{HZ},
Hall--Ho \cite{HallHo1,HallHo2}, Demni--Hamdi \cite{DemniHamdi}, and
Eaknipitsari--Hall \cite{EH}. See also the first author's expository
discussion of the PDE\ method \cite{PDE}.

We consider also $\partial S/\partial\varepsilon.$ We use the general formula
for the derivative of the trace of a logarithm,%
\[
\frac{d}{du}\mathrm{tr}[\log(a(u))]=\mathrm{tr}[a(u)^{-1}],
\]
whenever $a(\cdot)$ is a differentiable function with values in the space of
positive elements of $\mathcal{A}.$ (See \cite[Lemma 1.1]{Brown} or
\cite[Equation (25)]{PDE}.) Using this result, we compute that%
\begin{equation}
\frac{\partial S}{\partial\varepsilon}=\mathrm{tr}[(\left\vert a-\lambda
\right\vert ^{2}+\varepsilon)^{-1}]=\int_{0}^{\infty}\frac{1}{\xi+\varepsilon
}~d\mu_{\left\vert a-\lambda\right\vert ^{2}}(\xi).\label{regResolvent}%
\end{equation}
We then consider the behavior of $\partial S/\partial\varepsilon$ when
$\varepsilon$ tends to zero and try to understand what it tells us about the
Brown measure. We first note that,%
\begin{equation}
\lim_{\varepsilon\rightarrow0^{+}}\frac{\partial S}{\partial\varepsilon
}=\mathrm{tr}[\left\vert a-\lambda\right\vert ^{-2}]\label{limdSdEpsilon1}%
\end{equation}
where we \textit{define} the right-hand side of (\ref{limdSdEpsilon1}) as the limit
$\varepsilon\rightarrow0^{+}$ of the last expression in (\ref{regResolvent}),
namely (by monotone convergence)%
\begin{equation}
\mathrm{tr}[\left\vert a-\lambda\right\vert ^{-2}]=\int_{0}^{\infty}\frac
{1}{\xi}~d\mu_{\left\vert a-\lambda\right\vert ^{2}}(\xi
).\label{limdSdEpsilon2}%
\end{equation}

The quantity $\frac{\partial S}{\partial\varepsilon}(\lambda,\varepsilon)$
will typically blow up as $\varepsilon\rightarrow0^{+},$ when $\lambda$ is in
the support of the Brown measure of $a.$ We may consider for example,
\cite{DHK}, which computes the Brown measure of the free multiplicative
Brownian motion $b_{t}.$ In that setting, $\frac{\partial S}{\partial
\varepsilon}(\lambda,\varepsilon)$ blows up like $1/\sqrt{\varepsilon}$ for
$\lambda$ the interior of the support of the Brown measure of $b_{t},$ by
Proposition (5.6) and Equation (5.12) in \cite{DHK}.

We emphasize that (\ref{limdSdEpsilon2}) can have a finite value even if
$a-\lambda$ fails to have a \textit{bounded} inverse; it is enough for
$a-\lambda$ to have a inverse in the noncommutative $L^{2}$ space of operators
$b$ with $\mathrm{tr}[b^{\ast}b]<\infty.$ Thus, the condition that
$\frac{\partial S}{\partial\varepsilon}(\lambda,\varepsilon)$ has a finite
limit as $\varepsilon\rightarrow0^{+}$ does not, by itself, guarantee that
$\lambda$ is outside the spectrum of $a.$ On the other hand, the following
result of Zhong says that failure of $\partial S/\partial\varepsilon$ to blow
up near $\lambda_{0}$ indicates that $\lambda_{0}$ is outside the Brown
support of $a.$

\begin{theorem}
[Zhong]If
\[
\lim_{\varepsilon\rightarrow0^{+}}\frac{\partial S}{\partial\varepsilon
}(\lambda,\varepsilon)=\mathrm{tr}\left[  \left\vert a-\lambda\right\vert
^{-2}\right]
\]
is finite for all $\lambda$ in some neighborhood of $\lambda_{0},$ then
$\lambda_{0}$ is outside the support of the Brown measure of $a.$
\end{theorem}

See \cite[Theorem 4.6]{Zhong}. This result is a strengthening of a result of
Hall--Kemp \cite[Theorem 1.2]{HK}, which requires finiteness (and local
boundedness) of the quantity $\mathrm{tr}[\left\vert (a-\lambda)^{2}%
\right\vert ^{-2}].$

\subsection{The main result}

Our main result is a characterization of points $\lambda$ outside the spectrum
of $a$ as the points where $\frac{\partial S}{\partial\varepsilon}%
(\lambda,\varepsilon)$ extends analytically in $\varepsilon$ to a neighborhood
of $\varepsilon=0.$ 

\begin{theorem}
\label{main.thm}Fix an element $a$ in a tracial von Neumann algebra
$(\mathcal{A},\mathrm{tr})$ and define $S$ by (\ref{Sdef}). Suppose that for a
fixed $\lambda\in\mathbb{C},$ the function%
\begin{equation}
\varepsilon\mapsto\frac{\partial S}{\partial\varepsilon}(\lambda
,\varepsilon),\quad\varepsilon>0, \label{epsilonMapsTo}%
\end{equation}
admits a real-analytic extension from $\varepsilon\in(0,\infty)$ to
$\varepsilon\in(-\delta,\infty)$ for some $\delta>0.$ Then $a-\lambda$ is
invertible, meaning that $\lambda$ is outside the spectrum of $a.$

Conversely, if $\lambda$ is outside the spectrum of $a,$ the map in
(\ref{epsilonMapsTo}) admits a real-analytic extension to $(-\delta,\infty)$
for some $\delta>0.$
\end{theorem}

We abbreviate the condition in the theorem as \textquotedblleft$\partial
S/\partial\varepsilon$ is analytic in $\varepsilon$ at $\varepsilon
=0.$\textquotedblright\ Note that if, for $\lambda$ fixed, $S$ itself is
analytic in $\varepsilon$ at $\varepsilon=0,$ so is $\partial S/\partial
\varepsilon.$

In Section \ref{applications.sec}, we will give several examples where
$\partial S/\partial\varepsilon$ can be computed using the PDE method, giving
restrictions on the spectrum of $a.$ In many of these examples, we will show
that the spectrum of $a$ equals its Brown support.

\begin{remark}
\label{epsilonSquared.rem}Suppose instead of the function $S,$ we consider the
function
\[
\tilde{S}(\lambda,\varepsilon)=S(\lambda,\varepsilon^{2}),
\]
as in \cite{HallHo2}. Suppose, for a fixed $\lambda,$ we can show that
$\tilde{S}(\lambda,\varepsilon)$ has a real-analytic extension from
$\varepsilon\in(-\infty,0)$ to $\varepsilon\in(-\infty,\delta)$ for some
$\delta>0$ and that this extension is an \textbf{even} function of
$\varepsilon$ on $(-\delta,\delta).$ Then $S(\lambda,\varepsilon)=\tilde
{S}(\lambda,\sqrt{\varepsilon})$ will have also have a real-analytic extension
to a neighborhood of $\varepsilon=0.$

We emphasize, however, that the existence of a real-analytic extension of
$\tilde{S}(\lambda,\varepsilon)$ from $\varepsilon\in(-\infty,0)$ to
$\varepsilon\in(-\infty,\delta)$ \emph{does not}---without the evenness
assumption---guarantee that $\lambda$ is outside the spectrum of $a.$ Indeed,
Theorem 6.4 in \cite{DHK} shows that such a real-analytic extension of
$\tilde{S}$ can exist even for $\lambda$ in the Brown support of $a.$
\end{remark}

We need the following (presumably well-known) result for tracial von Neumann algebras.

\begin{lemma}
\label{invertibility.lem}For all $b\in(\mathcal{A},\tau),$ if $b^{\ast}b$ is
invertible, $b$ is also invertible.
\end{lemma}

Of course, this result does not hold for general operators on a Hilbert space.
If $b$ is an isometry that is not surjective, then $b^{\ast}b=1$ is invertible
but $b$ itself is not invertible. The lemma says that such examples cannot
occur in a tracial von Neumann algebra.

\begin{proof}
We use the polar decomposition to write $b=up,$ where $p$ is a non-negative
self-adjoint operator and $u$ is a partial isometry with the kernel of $u$
equal to the kernel of $p.$ Now, $b$ must be injective in order for $b^{\ast
}b$ to be invertible, and therefore $p$ must also be injective. Thus, $\ker
u=\ker p=\{0\}.$ We conclude that $u$ is actually an isometry: $u^{\ast}u=1.$

Meanwhile, it is known \cite[6.1.3 Proposition]{KadisonRingrose} that $u$ and
$p$ must also belong to $\mathcal{A}.$ Then by the cyclic property of the
trace,
\begin{equation}
\mathrm{tr}[1-uu^{\ast}]=\mathrm{tr}[1-u^{\ast}u]=\mathrm{tr}[1-1]=0.
\label{1minusa*a}%
\end{equation}
But $uu^{\ast}$ is the orthogonal projection onto the range of $u$. If this
range were not the whole Hilbert space, $1-uu^{\ast}$ would be a nonzero,
non-negative operator and (\ref{1minusa*a}) would contradict the faithfulness
of the trace. Thus, $u$ is unitary and therefore invertible. But $p=(b^{\ast
}b)^{1/2}$ is also invertible, so we conclude that $b=up$ is invertible.
\end{proof}

\begin{proof}
[Proof of Theorem \ref{main.thm}]We denote by $f$ the real-analytic extension
of the function in (\ref{epsilonMapsTo}), which is real analytic on
$(-\delta,\infty)$ for some $\delta>0.$ Then $f$ has a holomorphic extension,
also called $f,$ from $(-\delta^{\prime},\delta^{\prime})$ to an open disk
$D_{0}(\delta^{\prime})$ of radius $\delta^{\prime}$ centered at 0, for some
$\delta^{\prime}\leq\delta.$

Let $G$ denote the Cauchy transform of $\left\vert a-\lambda\right\vert ^{2},$
defined as%
\begin{equation}
G(z)=\mathrm{tr}\left[  (z-\left\vert a-\lambda\right\vert ^{2})^{-1}\right]
=\int_{0}^{\infty}\frac{1}{z-\xi}~d\mu_{\left\vert a-\lambda\right\vert ^{2}%
}(\xi), \label{CauchyDef}%
\end{equation}
which is a holomorphic function of $\lambda\in\mathbb{C}\setminus
\lbrack0,\infty).$ Note from (\ref{regResolvent}) that
\[
G(z)=-\frac{\partial S}{\partial\varepsilon}(\lambda,-z)=-f(-z)
\]
for $z<0.$ It follows that $G(z)$ agrees with $-f(-z)$ on the connected open
set $D_{0}(\delta^{\prime})\setminus\lbrack0,\delta^{\prime}).$ Thus, the
restriction of $G$ to $D_{0}(\delta^{\prime})\setminus\lbrack0,\delta^{\prime
})$ has a holomorphic extension to $D_{0}(\delta^{\prime}).$

Now, since $\frac{\partial S}{\partial\varepsilon}(\lambda,\varepsilon)$ is
real valued for $\varepsilon<0,$ its real-analytic extension $f$ is also real
valued on $(-\delta,\infty).$ Thus,%
\[
\lim_{y\rightarrow0^{+}}\operatorname{Im}[G(x+iy)]=-\lim_{y\rightarrow0^{+}%
}\operatorname{Im}[f(-x-iy)]=0
\]
for all $x\in(-\delta,\delta),$ where the limit is locally uniform in $x$ by
the continuity of $f$ on $D_{0}(\delta^{\prime}).$ Thus, by the Stieltjes
inversion formula, the measure $\mu_{\left\vert a-\lambda\right\vert ^{2}}$ is
zero on $(-\delta,\delta).$ Proposition \ref{normalSupport.prop} then tells us
that the spectrum of $\left\vert a-\lambda\right\vert ^{2}$ does not include 0.

In the opposite direction, if $\lambda$ is outside the spectrum of $a,$ then
$a-\lambda$ is invertible, so that $\left\vert a-\lambda\right\vert
^{2}=(a-\lambda)^{\ast}(a-\lambda)$ is also invertible. Then for all
$\varepsilon\in\mathbb{R}$ with $\left\vert \varepsilon\right\vert
<\delta:=\left\Vert \left\vert a-\lambda\right\vert ^{-2}\right\Vert ,$ the
inverse of $\left\vert a-\lambda\right\vert ^{2}+\varepsilon$ exists, with%
\begin{align*}
(\left\vert a-\lambda\right\vert ^{2}+\varepsilon)^{-1}  &  =\left\vert
a-\lambda\right\vert ^{2}(1+\varepsilon\left\vert a-\lambda\right\vert
^{-2})\\
&  =\left\vert a-\lambda\right\vert ^{2}\sum_{k=0}^{\infty}(-1)^{k}%
\varepsilon^{k}\left\vert a-\lambda\right\vert ^{-2k}.
\end{align*}
Applying the trace to this relation gives a real-analytic function on
$(-\delta,\delta)$ that agrees with $\frac{\partial S}{\partial\varepsilon
}(\lambda,\varepsilon)$ on $(0,\delta).$
\end{proof}

\section{Applications\label{applications.sec}}

In the remainder of the paper, we study several examples where Theorem
\ref{main.thm} can be used to give information about the spectrum of certain
elements. The examples are mostly ones in which the PDE\ method introduced by
Driver--Hall--Kemp \cite{DHK} is used, such as papers by Hall--Ho
\cite{HallHo1,HallHo2}, Ho \cite{HoElliptic}, Ho--Zhong \cite{HZ},
Demni--Hamdi \cite{DemniHamdi}, and Eaknipitsari--Hall \cite{EH}. We also
analyze the examples studied by Zhong \cite{Zhong}, where Zhong uses
free-probability techniques instead of the PDE\ method but gets formulas
similar to what one obtains from the PDE\ method.

In most cases, we show that the spectrum equals the Brown support, showing
that the PDE\ method is even more powerful than was previously recognized.

We divide the examples into two broad classes, which we refer as
\textquotedblleft additive\textquotedblright\ and \textquotedblleft
multiplicative.\textquotedblright

\subsection{Additive case\label{addStatements.sec}}

A \textbf{semicircular} element $x_{t}$ of variance $t>0$ in a tracial von
Neumann algebra is a self-adoint element whose law is the semicircular measure
on $[-2\sqrt{t},2\sqrt{t}]$, i.e., the measure with density $\frac{1}{2\pi
t}\sqrt{4t-x^{2}}$ on this interval. A \textbf{circular} element $c_{t}$ of
variance $t$ is then an element of the form%
\begin{equation}
c_{t}=\frac{1}{\sqrt{2}}(x_{t}+iy_{t}), \label{ctDef}%
\end{equation}
where $x_{t}$ and $y_{t}$ are freely independent elements of variance $t.$ The
Brown measure of $c_{t}$ is the uniform probability measure on a disk of
radius $\sqrt{t}.$

An \textbf{elliptic} element\ is then an element of the form%
\begin{equation}
g=e^{i\theta}(ax+iby), \label{gdef}%
\end{equation}
where $a,$ $b,$ and $\theta$ are real numbers, with $a$ and $b$ not both zero,
and where $x$ and $y$ are freely independent semicircular elements of variance
1. The $\ast$-distribution of $g$ is determined by the positive real number
$t$ given by%
\begin{equation}
t=\mathrm{tr}[g^{\ast}g] \label{tDef}%
\end{equation}
and the complex number $\gamma$ given by%
\begin{equation}
\gamma=\mathrm{tr}[g^{2}]. \label{gammaDef}%
\end{equation}
Then $\gamma$ satisfies%
\begin{equation}
\left\vert \gamma\right\vert \leq t \label{gammaAndT}%
\end{equation}
and any pair $t>0$ and $\gamma\in\mathbb{C}$ satisfying (\ref{gammaAndT})
arises for some choice of $a,$ $b,$ and $\theta.$ (See \cite[Section
2.1]{HallHo2} or \cite[Section 2.4]{Zhong}.) We use the notation $g_{t,\gamma
}$ to denote such an element.

The case $\gamma=0$ corresponds to the case $a=b$ in (\ref{gdef}), in which
case $g_{t,\gamma}=g_{t,0}$ is a circular element of variance $t.$ The case in
which $\gamma=t$ corresponds to $\theta=b=0$ in (\ref{gdef}) and gives a
semicircular element of variance $a^{2}.$ We refer to models involving
elliptic elements as \textquotedblleft additive,\textquotedblright\ since the
sum of two freely independent elliptic elements is again elliptic.
Specifically, if $g_{t_{1},\gamma_{1}}$ and $g_{t_{2},\gamma_{2}}$ are freely
independent elliptic elements, then%
\begin{equation}
g_{t_{1},\gamma_{1}}+g_{t_{2},\gamma_{2}}\overset{d}{=}g_{t_{1}+t_{2}%
,\gamma_{1}+\gamma_{2}}, \label{sumElliptic}%
\end{equation}
where $\overset{d}{=}$ denotes equality in $\ast$-distribution. (Compare
(\ref{multElliptic}) in the multiplicative case.)

Ho and Zhong \cite[Section 3]{HZ} computed the Brown measure of an element of
the form $x+c_{t},$ where $x$ is self-adjoint and freely independent of
$c_{t}.$ See Figure \ref{bernoulli.fig}. Zhong then computed the Brown measure
of $x+c_{t},$ where $x$ is freely independent of $c_{t}$ but otherwise
arbitrary. See Figure \ref{trinoulli.fig}.%

\begin{figure}[ptb]
\centering
\includegraphics[scale=0.55]{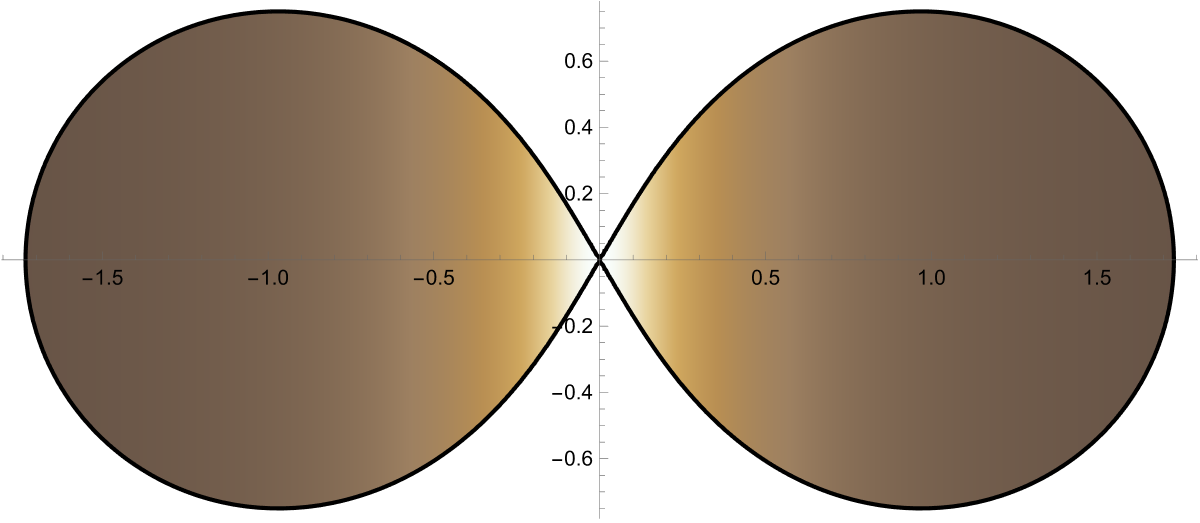}
\caption{Density plot of the Brown measure of $x+c_{t}$, where $x$ is
self-adjoint with $\mu_{x}=\frac{1}{2}(\delta_{-1}+\delta_{1}),$ with $t=1.$}
\label{bernoulli.fig}
\end{figure}

\begin{figure}[ptb]
\centering
\includegraphics[scale=0.4]{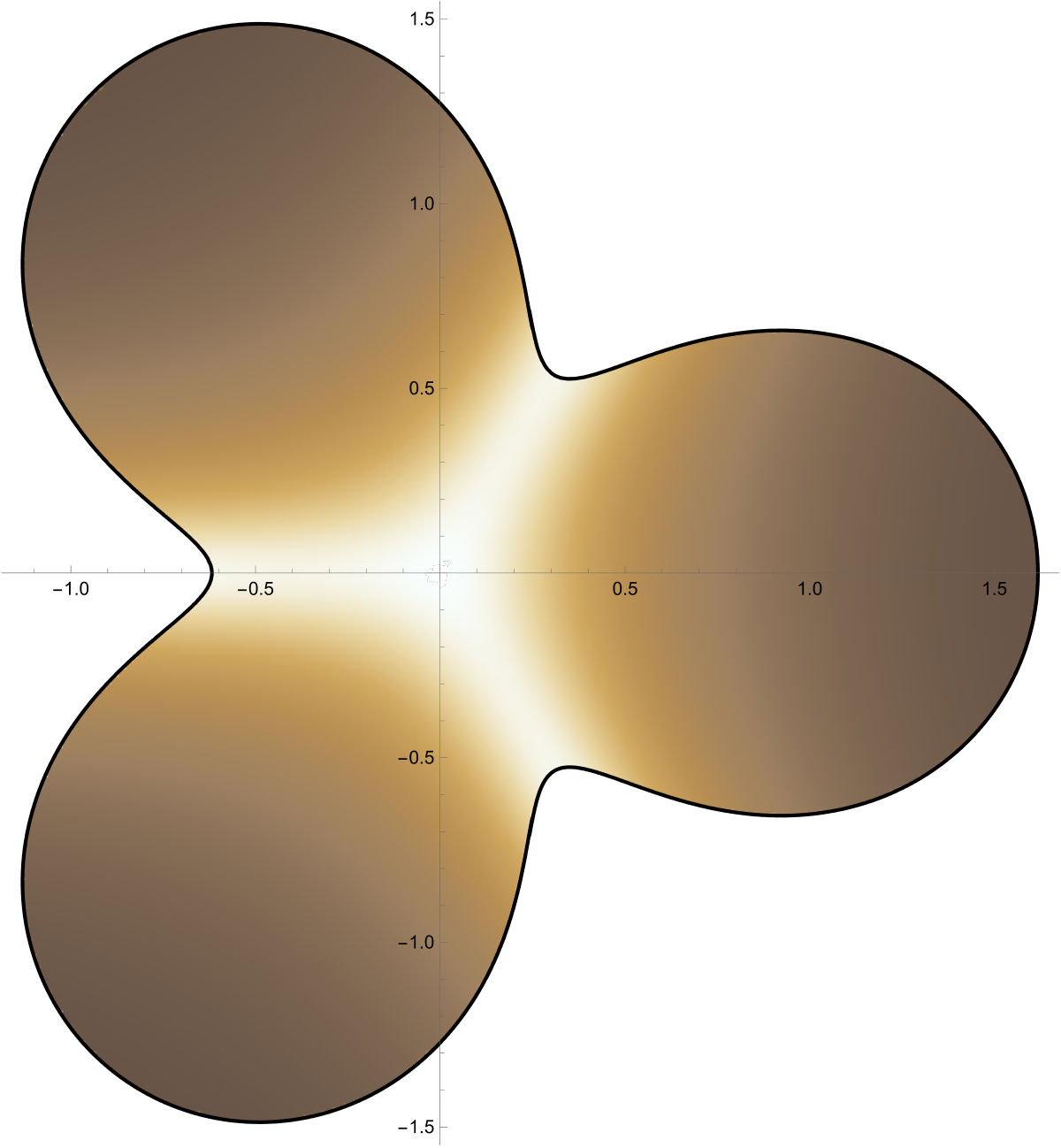}
\caption{Density plot of the Brown measure of $x+c_{t},$ where $x$ is unitary
and $\mu_{x}$ is supported at the third roots of unity, with equal masses, and
$t=1.$}
\label{trinoulli.fig}
\end{figure}

Meanwhile, Hall and Ho \cite{HallHo1} computed the Brown measure of an element
of the form $x+iy,$ where $x$ is self-adjoint, $y$ is semicircular, and $x$
and $y$ are freely independent. Ho \cite{HoElliptic} then computed the Brown
measure of $x+g_{t,\gamma},$ where $x$ is self-adjoint, $g_{t,\gamma}$ is
elliptic with $\gamma\in\mathbb{R}$, and $x$ and $g_{t,\gamma}$ are freely
independent. See Figure \ref{bernoullipush.fig}. Finally, Zhong \cite{Zhong}
computed the Brown measure of $x+g_{t,\gamma}$ where $x$ is arbitrary,
$g_{t,\gamma}$ is a general elliptic element, and $x$ and $g_{t,\gamma}$ are
freely independent.

\begin{figure}[ptb]
\centering
\includegraphics[scale=0.7]{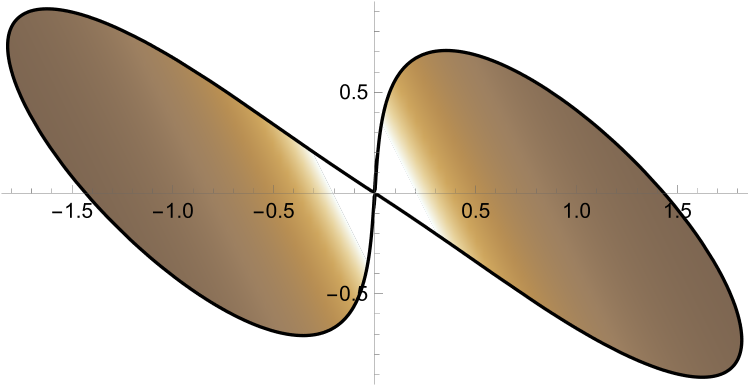}
\caption{Density plot of the Brown measure of $x+g_{t,\gamma}$ with $x$
self-adjoint and $\mu_{x}=\frac{1}{2}(\delta_{-1}+\delta_{1}),$ with $t=1$ and
$\gamma=-i/2.$}
\label{bernoullipush.fig}
\end{figure}

The papers of Ho--Zhong \cite{HZ}, Hall--Ho \cite{HallHo1}, and Ho
\cite{HoElliptic} are based on the PDE\ method introduced in \cite{DHK}. The
paper of Zhong \cite{Zhong}, by contrast, uses techniques of free probability
and subordination functions. Nevertheless, the formulas obtained by Zhong are
parallel to what one gets when using the PDE\ method.

\begin{theorem}
\label{add1.thm}Let $g_{t,\gamma}$ be an elliptic element with parameters $t$
and $\gamma$ as in (\ref{tDef}) and (\ref{gammaDef}) and let $x$ be a
\emph{self-adjoint} element that is freely independent of $g_{t,\gamma}.$ Then
the spectrum and the Brown support of $x+g_{t,\gamma}$ coincide:%
\[
\sigma(x+g_{t,\gamma})=\mathrm{supp}(\mathrm{Br}_{x+g_{t,\gamma}}).
\]

\end{theorem}

The preceding result does not hold if $x$ is a arbitrary element freely
independent of $g_{t,\gamma},$ even in the case that $g_{t,\gamma}$ is
circular, as the following example shows.

\begin{example}
\label{specBrownSupport.ex}Suppose $h$ is a non-negative self-adjoint element
such that (1) the spectrum of $h$ contains $0$ and (2) $h$ has an $L^{2}$
inverse, meaning that
\[
\mathrm{tr}\left[  h^{-2}\right]  :=\int_{0}^{a}\frac{1}{\xi^{2}}~d\mu_{h}%
(\xi)<\infty.
\]
Let $x=uh,$ where $u$ is a Haar unitary that is freely independent of $h.$
Then by Example \ref{hl.ex}, the Brown support of $x$ is a proper subset of the spectrum of $x$. Furthermore, for all sufficiently small $t,$ the Brown support of $x+c_t$ is a proper subset of the spectrum of
$x+c_{t}$.
\end{example}

The proof of this result is deferred to Section \ref{exampleCheck.sec}.

A natural assumption to impose on $x$ is that the desired result (equality of
spectrum and Brown support) should hold for $x$ itself.

\begin{theorem}
\label{add2.thm}Let $g_{t,\gamma}$ be an elliptic element with parameters $t$
and $\gamma$ as in (\ref{tDef}) and (\ref{gammaDef}) and let $x$ be an element
that is freely independent of $g_{t,\gamma}$ but not necessarily self-adjoint.
Assume that $\sigma(x)=\mathrm{supp}(\mathrm{Br}_{x}),$ which will hold, for
example, if $x$ is normal (Proposition \ref{normalSupport.prop}). Then
\[
\sigma(x+g_{t,\gamma})=\mathrm{supp}(\mathrm{Br}_{x+g_{t,\gamma}}).
\]

\end{theorem}

For more detailed statements of the preceding results, see Corollary
\ref{equalSupports.cor} in the circular case and Theorem
\ref{ellipticEquality.thm} in the general elliptic case.

For general $x,$ we can still prove the desired result, provided that $t$ is
large enough.

\begin{theorem}
\label{add3.thm}Let $g_{t,\gamma}$ be an elliptic element with parameters $t$
and $\gamma$ as in (\ref{tDef}) and (\ref{gammaDef}) and let $x$ be an element
that is freely independent of $g_{t,\gamma}$ but not necessarily self-adjoint.
Then for all sufficiently large $t>0,$ we have
\[
\sigma(x+g_{t,\gamma})=\mathrm{supp}(\mathrm{Br}_{x+g_{t,\gamma}})
\]
for all $\gamma\in\mathbb{C}$ with $\left\vert \gamma\right\vert \leq t.$
\end{theorem}

See Corollary \ref{largeT.cor}.

\begin{remark}
In Theorems \ref{add1.thm}, \ref{add2.thm}, and \ref{add3.thm}, we may take
$\gamma=0,$ in which case $g_{t,\gamma}$ becomes a circular element of
variance $t$ and we obtain%
\[
\sigma(x+c_{t})=\mathrm{supp}(\mathrm{Br}_{x+c_{t}}),
\]
under the stated hypotheses.
\end{remark}

\begin{remark}
The proofs of the preceding theorems rely on the prior computation of the
relevant Brown support as the closure of a certain domain. Our new
contribution is to show that there is no spectrum outside the closure of the
domain. Then since the Brown measure of any element is supported on its
spectrum, we obtain equality of the spectrum and Brown support.

In the circular case, $\mathrm{supp}(\mathrm{Br}_{x+c_{t}})$ was computed---by
Ho--Zhong \cite[Section 3]{HZ} when $x$ is self-adjoint and by Zhong
\cite{Zhong} when $x$ is arbitrary---as the closure $\overline{\Sigma}_{t}$ of
a certain domain $\Sigma_{t}.$ Then by results of \cite{Zhong}, $\mathrm{supp}%
(\mathrm{Br}_{x+g_{t,\gamma}})$ is the image of $\mathrm{supp}(\mathrm{Br}%
_{x+c_{t}})$ under a certain explicit map $\Phi_{t,\gamma}.$ See Section
\ref{arbPlusElliptic.sec} for more information.
\end{remark}

The preceding theorems will be proven in Section \ref{addProofs.sec}, in the
following stages. We will start by analyzing $x+c_{t}$ in the self-adjoint
case and then extend the arguments to $x+c_{t}$ where $x$ is not self-adjoint.
Finally, for general $x,$ we will connect the case $x+g_{t,\gamma}$ to the
case $x+c_{t}.$

\subsection{Multiplicative case\label{mult1.sec}}

We begin by giving a nonrigorous motivation for the model will introduce.
Using (\ref{sumElliptic}), we can see that, for any $k,$
\begin{equation}
g_{t,\gamma}\overset{d}{=}\frac{g_{t,\gamma}^{1}}{\sqrt{k}}+\cdots
+\frac{g_{t,\gamma}^{k}}{\sqrt{k}}, \label{gg}%
\end{equation}
where $g_{t,\gamma}^{1},\ldots,g_{t,\gamma}^{k}$ are freely independent copies
of $g_{t,\gamma}.$ We then make a \textquotedblleft
multiplicative\textquotedblright\ model by exponentiating, but where in the
noncommutative setting, the correct way to exponentiate is to exponentiate the
terms on the right-hand side of (\ref{gg}) separately and then multiply the
results. Thus, we may consider
\begin{equation}
\exp\left\{  \frac{ig_{t,\gamma}^{1}}{\sqrt{k}}\right\}  \exp\left\{
\frac{ig_{t,\gamma}^{2}}{\sqrt{k}}\right\}  \ldots\exp\left\{  \frac
{ig_{t,\gamma}^{k}}{\sqrt{k}}\right\}  . \label{zerothApprox}%
\end{equation}
Here the factor $i$ in the exponent is just a convention, which will give a
nicer match of the parameters between the additive and multiplicative cases.
(Note that $ig_{t,\gamma}$ is again an elliptic element, with parameters $t$
and $-\gamma.$)

For large $k,$ we may reasonably hope to approximate each exponential in
(\ref{zerothApprox}) using the Taylor series of the exponential through the
quadratic order, considering instead%
\begin{equation}
\left(  1+\frac{ig_{t,\gamma}^{1}}{\sqrt{k}}-\frac{(g_{t,\gamma}^{1})^{2}}%
{2k}\right)  \left(  1+\frac{ig_{t,\gamma}^{2}}{\sqrt{k}}-\frac{(g_{t,\gamma
}^{2})^{2}}{2k}\right)  \cdots\left(  1+\frac{ig_{t,\gamma}^{k}}{\sqrt{k}%
}-\frac{(g_{t,\gamma}^{k})^{2}}{2k}\right)  . \label{firstApprox}%
\end{equation}
Now, for large $k,$ the term involving $(g_{t,\gamma}^{j})^{2}$ will be
smaller than the term involving $g_{t,\gamma}^{j}$, because it has $k$ rather
than $\sqrt{k}$ in the denominator. Nevertheless, the $(g_{t,\gamma}^{j})^{2}$
term is not negligible compared to the $g_{t,\gamma}^{j}$ term, because the
$g_{t,\gamma}^{j}$ term has mean zero, while the $(g_{t,\gamma}^{j})^{2}$
generally has nonzero mean. We expect, however, that we can replace
$(g_{t,\gamma}^{j})^{2}$ by $\mathrm{tr}[(g_{t,\gamma}^{j})^{2}]=\gamma$ in
the large-$k$ limit, giving another model that should have the same large-$k$
behavior:%
\begin{equation}
\left(  1+\frac{ig_{t,\gamma}^{1}}{\sqrt{k}}-\frac{\gamma}{2k}\right)  \left(
1+\frac{ig_{t,\gamma}^{2}}{\sqrt{k}}-\frac{\gamma}{2k}\right)  \cdots\left(
1+\frac{ig_{t,\gamma}^{k}}{\sqrt{k}}-\frac{\gamma}{2k}\right)  .
\label{secondApprox}%
\end{equation}
To motivate the change from (\ref{firstApprox}) to (\ref{secondApprox}), we
can compute that for a natural random matrix approximation $g_{t,\gamma}^{N}$
to $g_{t,\gamma},$ we have
\[
\mathbb{E}\{(g_{t,\gamma}^{N})^{2}\}=\mathbb{E}\{\mathrm{tr}[(g_{t,\gamma
,}^{N})^{2}]\}I,
\]
where, here, $\mathrm{tr}$ denotes the normalized trace of a matrix.

At the rigorous level, we may define an elliptic Brownian motion $w_{t,\gamma
}(r)$ by replacing the semicircular elements $x$ and $y$ in (\ref{gdef}) by
\textbf{semicircular Brownian motions} $x_{r}$ and $y_{r},$ that is,
continuous processes with freely independent, semicircular increments. Then we
may consider a free stochastic differential equation (\ref{fSDE}) based on
(\ref{secondApprox}) as%
\begin{equation}
db_{t,\gamma}(r)=b_{t}\left(  i~dw_{t,\gamma}(r)-\frac{\gamma}{2}~dr\right)
,\quad b_{t,\gamma}(0)=1, \label{fSDE}%
\end{equation}
where the $dr$ term is an It\^{o} correction. We then define the \textbf{free
multiplicative Brownian motion} with parameters $t$ and $\gamma$ as
$b_{t,\gamma}$ as the value of $b_{t,\gamma}(r)$ at $r=1$:%
\begin{equation}
b_{t,\gamma}=\left.  b_{t,\gamma}(r)\right\vert _{r=1}. \label{btgamma}%
\end{equation}
See Section 2.1 in \cite{HallHo2} for more information, where the parameter
$s$ in \cite{HallHo2} corresponds to $t$ here, while the parameter $\tau$ in
\cite{HallHo2} corresponds to $t-\gamma$ here. We refer to $b_{t,\gamma}$ as a
\textquotedblleft multiplicative\textquotedblright\ model, since it satisfies
the multiplicative counterpart of (\ref{sumElliptic}):%
\begin{equation}
b_{t_{1},\gamma_{1}}b_{t_{2},\gamma_{2}}\overset{d}{=}b_{t_{1}+t_{2}%
,\gamma_{1}+\gamma_{2}.} \label{multElliptic}%
\end{equation}
See Theorem 4.3 in \cite{HallHo2}.

The expression (\ref{zerothApprox}) represents a Wong--Zakai approximation (as
in \cite{WZ} or \cite{Tw}) to $b_{t,\gamma}(1)$, obtained by making a
piecewise-linear approximation $w_{t,\gamma}^{(k)}$ to the Brownian motion
$w_{t,\gamma}$ and then solving (\ref{fSDE}) with $w_{t,\gamma}$ replaced by
$w_{t,\gamma}^{(k)}$ (but without the It\^{o} correction term). The expression
(\ref{secondApprox}) then represents a more numerically tractable
approximation to $b_{t,\gamma}(1).$

In the case $\gamma=0,$ the free multiplicative Brownian motion was introduced
by Biane. See \cite[Section 4.2]{BianeJFA}, where what we are calling the free
multiplicative Brownian motion (with $\gamma=0$) is denoted $\Lambda_{t}.$ We
use a special notation for the $\gamma=0$ case:
\begin{equation}
b_{t}=b_{t,0}. \label{bt}%
\end{equation}
Meanwhile, the case $\gamma=t$ corresponds to Biane's free unitary Brownian
motion $u_{t}$, introduced in \cite{BianeFields}:%
\[
b_{t,t}=u_{t}.
\]
The general form of the free multiplicative Brownian motion was introduced by
Hall and Ho in \cite{HallHo2}, where $s$ and $\tau$ in \cite{HallHo2}
corresponds to $t$ and $t-\gamma,$ respectively, here.

Hall and Kemp showed that the support of Brown measure of $b_{t}$ is contained
in the closure of a certain set $\Sigma_{t}$, which was introduced by Biane in
\cite[Section 4.2.6]{BianeJFA}. Driver, Hall, and Kemp \cite{DHK} then
computed the Brown measure of $b_{t}$ and showed that its support is exactly
$\overline{\Sigma}_{t}.$ Ho and Zhong \cite[Section 4]{HZ} extended the
results of \cite{DHK} to compute the Brown measure of $ub_{t},$ where $u$ is a
unitary element that is freely independent of $b_{t}$. Finally, Hall and Ho
computed the Brown measure of $ub_{t,\gamma}$ for general $t$ and $\gamma.$
See Figures \ref{quad.fig} and \ref{quadpush.fig}.

Meanwhile, Demni and Hamdi \cite{DemniHamdi} studied the unitary Brownian
motion $u_{t}$ multiplied by a \textit{non-negative self-adjoint} initial
condition $x$ (freely independent of $u_{t}$). In the case that $x$ is a
self-adjoint projection, they identified a natural domain $\Sigma_{t}$ and
showed that the support of the Brown measure of $xu_{t}$ is contained in
$\{0\}\cup\overline{\Sigma}_{t}.$ Eaknipitsari and Hall \cite{EH} then
extended the results of \cite{DemniHamdi} to the case of $xb_{t,\gamma},$
where $x$ is a non-negative self-adjoint element freely independent of
$b_{t,\gamma}.$

We now obtain information about the spectrum of $ub_{t,\gamma}$ and
$xb_{t,\gamma}.$%

\begin{figure}[ptb]
\centering
\includegraphics[scale=0.6]{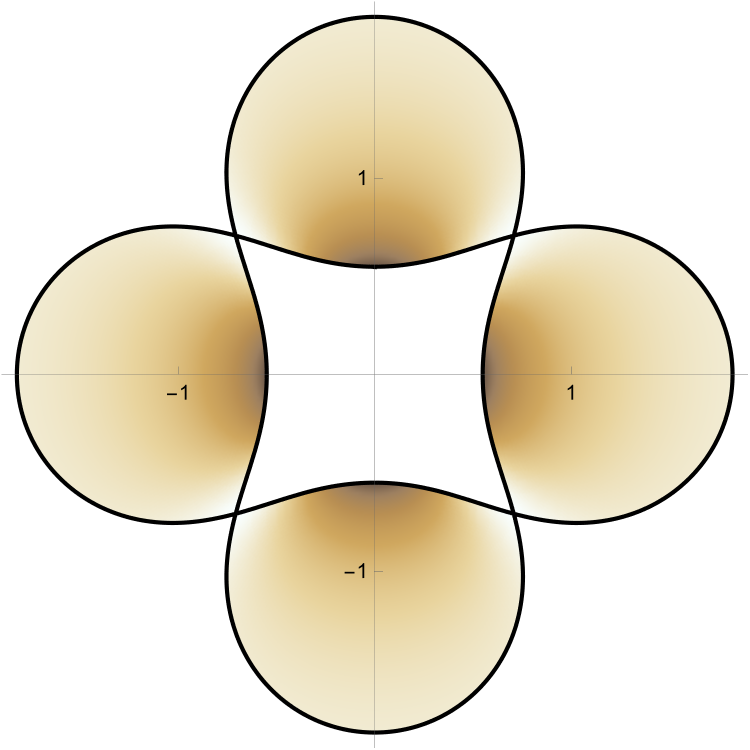}
\caption{Density plot of Brown measure of $ub_{t},$ where $u$ is unitary with
$\mu_{u}$ supported at the fourth roots of unity, with equal mass, for $t=1.$}
\label{quad.fig}
\end{figure}

\begin{theorem}
Let $b_{t}$ be the free multiplicative Brownian motion in (\ref{bt}) and let
$u$ be a unitary element that is freely independent of $b_{t}.$ Then for all
$t>0,$ we have%
\[
\sigma(ub_{t})=\mathrm{supp}(\mathrm{Br}_{ub_{t}}).
\]
More generally, for any $t>0$ and $\gamma\in\mathbb{C}$ with $\left\vert
\gamma\right\vert \leq t,$ we have%
\[
\sigma(ub_{t,\gamma})=\mathrm{supp}(\mathrm{Br}_{ub_{t,\gamma}}).
\]

\end{theorem}

For more detailed statements, see Theorem \ref{ubtSpec.thm} in the case of
$ub_{t}$ and Theorem \ref{multEquality.thm} in the case of $ub_{t,\gamma}.$ As
in the additive case, the proof of the theorem relies on the prior computation
of the Brown support, in \cite{DHK} for the case of $b_{t}$ itself, in
\cite[Section 4]{HZ} for $ub_{t},$ and in \cite{HallHo2} for $ub_{t,\gamma}$.

The case of a non-negative initial condition $x$ \cite{DemniHamdi,EH} is
conceptually similar to the case of a unitary initial condition, but more
algebraically complicated. The algebraic complications prevent a rigorous
computation of the Brown measure of $xb_{t,\gamma}.$ But \cite{EH} shows that
the \textit{support} of the Brown measure of $xb_{t,\gamma}$ is contained in
$\{0\}\cup D_{t,\gamma}$ for a certain closed set $\gamma.$ We then show that
\textquotedblleft most\textquotedblright\ points outside $\{0\}\cup
D_{t,\gamma}$ are outside the spectrum of $xb_{t,\gamma}.$ Precise statements
may be found in Section \ref{positive.sec}.%

\begin{figure}[ptb]
\centering
\includegraphics[scale=0.6]{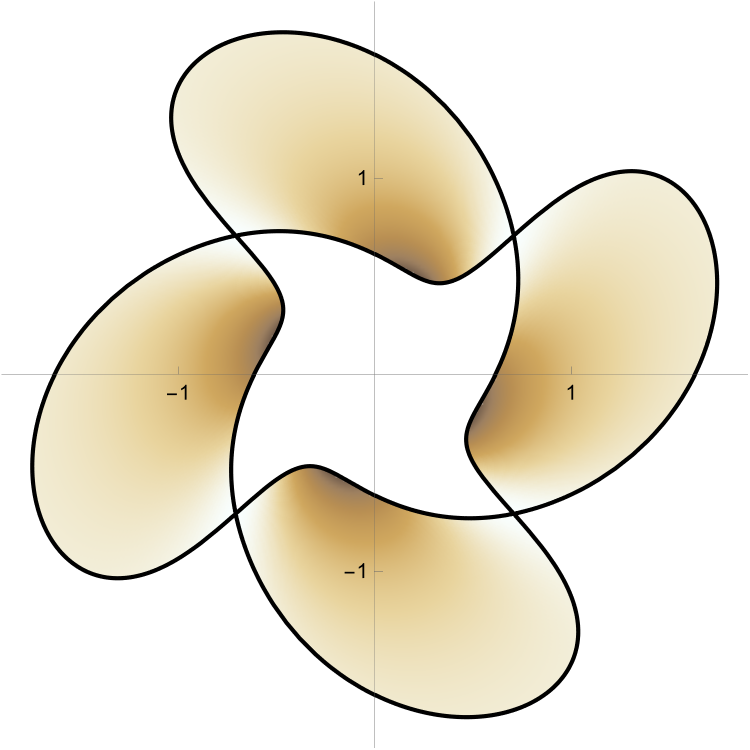}
\caption{Density plot of Brown measure of $ub_{t,\gamma},$ where $u$ is
unitary with $\mu_{u}$ supported at the fourth roots of unity, with equal
mass, for $t=1$ and $\gamma=-i/2.$}
\label{quadpush.fig}
\end{figure}

\subsection{Verification of Example \ref{specBrownSupport.ex}%
\label{exampleCheck.sec}}

Let $x=uh$ be as in the example, in which case Example \ref{hl.ex} applies to
$x.$ We will show that Example \ref{hl.ex} continues to apply to $x+c_{t},$
for sufficiently small $t.$ To do this, we will show that $0$ is in the
spectrum of $\left\vert x+c_{t}\right\vert $ but that $\left\vert
x+c_{t}\right\vert $ has an $L^{2}$ inverse.

For a measure $\mu$ on $\mathbb{R},$ we let $\tilde{\mu}$ be the
symmetrization of $\mu,$ that is, the average of $\mu$ and its push-forward
under the map $\xi\mapsto-\xi.$ We first note that the symmetrization of
$\mu_{|c_{t}|}$ is the semicircular law $\mathrm{sc}_{t}$ of variance $t$
\cite[p. 174]{MingoSpeicher}. By Proposition 3.5 of \cite{HaagerupLarsen},
$x+c_{t}$ is $R$-diagonal and the law of $|x+c_{t}|$, determined by its
symmetrization, is the free convolution
\begin{equation}
\tilde{\mu}_{|x+c_{t}|}=\tilde{\mu}_{h}\boxplus\mathrm{sc}_{t}%
.\label{RdiagSum}%
\end{equation}
Now, by our assumption on $h$ and Proposition \ref{normalSupport.prop}, 0 is
in the support of $\mu_{h}$ and therefore also in the support of $\tilde{\mu
}_{h}.$ It then follows from Proposition 2.2 in \cite{Cap} that 0 is in the
support of $\tilde{\mu}_{h}\boxplus\mathrm{sc}_{t}$ for all $t>0.$ (Taking
$x=0$ in the cited proposition, the symmetry of the measures involved means
that the quantity $u$ in the proposition must be zero. Then since 0 is in the
support of $\tilde{\mu}_{h},$ the proposition tells us that $0$ is in the
support of $\tilde{\mu}_{h}\boxplus\mathrm{sc}_{t}.$) Thus, by (\ref{RdiagSum}%
), 0 is in the support of $\tilde{\mu}_{|x+c_{t}|}$ and thus, also, in the
support of $\mu_{|x+c_{t}|}.$ Using Proposition \ref{normalSupport.prop}
again, we conclude that 0 is in the spectrum of $\left\vert x+c_{t}\right\vert
.$

We now show that $\left\vert x+c_{t}\right\vert $ has an $L^{2}$ inverse, for
sufficiently small $t,$ using results of Biane \cite{BianeFreeHeat} computing
measures of the form $\mu\boxplus\mathrm{sc}_{t}.$ Consider the function
$v_{t}$ defined by%
\[
v_{t}(x)=\inf_{u\geq0}\left\{  y\geq0\left\vert \int_{\mathbb{R}}\frac
{1}{(x-\xi)^{2}+y^{2}}~d\tilde{\mu}_{h}(\xi)\leq\frac{1}{t}\right.  \right\}
,
\]
where $v_{t}$ is continuous by \cite[Lemma 2]{BianeFreeHeat}. Then
\begin{align}
v_{t}(x)  &  =0\Longleftrightarrow\int_{\mathbb{R}}\frac{1}{(x-\xi)^{2}%
}~d\tilde{\mu}_{h}(\xi)\leq\frac{1}{t}\nonumber\\
&  \Longleftrightarrow t\leq\left\Vert h^{-1}\right\Vert _{2}^{-2}.
\label{vtZero}%
\end{align}

We then define an open set $\Omega_{t}$ inside the upper half-plane as the set
of $x+iy$ with $y>v_{t}(x)$. Then define a holomorphic function on the upper
half-plane by%
\[
H_{t}(z)=z+tG_{\tilde{\mu}_{h}}(z),
\]
where $G_{\mu}$ denotes the Cauchy transform of a measure $\mu.$ According to
Lemma 4 of \cite{BianeFreeHeat}, $H_{t}$ maps $\Omega_{t}$ injectively onto
the upper half-plane. Furthermore, by (\ref{RdiagSum}) and \cite[Proposition
2]{BianeFreeHeat}, we have that
\begin{equation}
G_{\tilde{\mu}_{|x+c_{t}|}}(H_{t}(z))=G_{\tilde{\mu}_{h}}(z)
\label{bianeRelation}%
\end{equation}
for all $z\in\Omega_{t}.$

Now, since $\tilde{\mu}_{h}$ is symmetric,
\begin{equation}
G_{\tilde{\mu}_{h}}(i\varepsilon)=\frac{1}{2}\int_{\mathbb{R}}\left(  \frac
{1}{i\varepsilon-\xi}+\frac{1}{i\varepsilon+\xi}\right)  ~d\tilde{\mu}_{h}%
(\xi)=-i\varepsilon\int_{\mathbb{R}}\frac{1}{\xi^{2}+\varepsilon^{2}}%
d\tilde{\mu}_{h}(\xi), \label{gmuiEpsilon}%
\end{equation}
so that, by monotone convergence,
\[
\lim_{\varepsilon\rightarrow0^{+}}\frac{G_{\tilde{\mu}_{h}}(i\varepsilon
)}{-i\varepsilon}=\int_{\mathbb{R}}\frac{1}{\xi^{2}}d\tilde{\mu}_{h}%
(\xi)=\left\Vert x^{-1}\right\Vert _{2}^{2}.
\]

Using (\ref{gmuiEpsilon}), we calculate that
\[
H_{t}(i\varepsilon)=i\varepsilon\left(  1-t\int_{\mathbb{R}}\frac{1}{\xi
^{2}+\varepsilon^{2}}d\tilde{\mu}_{h}(\xi)\right)  .
\]
Since $\int_{\mathbb{R}}\frac{1}{\xi^{2}}~d\tilde{\mu}_{h}(\xi)$ is finite by
assumption, we see that $H_{t}(i\varepsilon)\rightarrow0$ as $\varepsilon
\rightarrow0^{+}.$ Furthermore, for $t<\Vert x^{-1}\Vert_{2}^{-2}$, we see
from (\ref{vtZero}) that $i\varepsilon$ is in $\Omega_{t}$ for all
$\varepsilon>0.$ Thus, using (\ref{bianeRelation})\ and letting $\varepsilon
\rightarrow0^{+},$ we obtain%
\begin{align*}
\Vert(x+c_{t})^{-1}\Vert_{2}^{2}  &  =\lim_{\varepsilon\rightarrow0^{+}}%
\frac{G_{\tilde{\mu}_{|x+c_{t}|}}(H_{t}(i\varepsilon))}{-H_{t}(i\varepsilon
)}\\
&  =\lim_{\varepsilon\rightarrow0^{+}}\frac{G_{\tilde{\mu}_{h}}(i\varepsilon
)}{-i\varepsilon}\frac{i\varepsilon}{H_{t}(i\varepsilon)}\\
&  =\lim_{\varepsilon\rightarrow0^{+}}\frac{G_{\tilde{\mu}_{h}}(i\varepsilon
)}{-i\varepsilon}\frac{1}{1-t\frac{G_{\tilde{\mu}_{h}}(i\varepsilon
)}{-i\varepsilon}}\\
&  =\Vert x^{-1}\Vert_{2}^{2}\frac{1}{1-t\Vert x^{-1}\Vert_{2}^{2}}.
\end{align*}
This shows that the inner radius of the Brown support of $x+c_{t}$ is
\[
\Vert(x+c_{t})^{-1}\Vert_{2}^{-1}=\sqrt{\Vert x^{-1}\Vert_{2}^{-2}-t},
\]
which is positive for all $t<\Vert x^{-1}\Vert_{2}^{-2}$.

\section{Additive case\label{addProofs.sec}}

In this section, we provide more detailed statements and proofs for the
results stated in Section \ref{addStatements.sec}.

\subsection{The self-adjoint plus circular case\label{saPlusCirc.sec}}

Ho and Zhong \cite[Section 3]{HZ} compute the Brown measure of an element of
the form $x+c_{t},$ where $x$ is self-adjoint, $c_{t}$ is circular of variance
$t,$ and $x$ and $c_{t}$ are freely independent. We now introduce some of the
PDE techniques from \cite{HZ} that we will use to prove a result about the
spectrum of $x+c_{t}.$ (See also \cite{PDE} for a gentle introduction to the
PDE method.) We consider the regularized log potential of $x+c_{t}$, as in
(\ref{sDef}), which we write as
\begin{equation}
S(t,\lambda,\varepsilon)=\mathrm{tr}[\log(\left\vert x+c_{t}-\lambda
\right\vert ^{2}+\varepsilon)],\quad\lambda\in\mathbb{C},~\varepsilon>0.
\label{Sadditive}%
\end{equation}
According to \cite[Proposition 3.2]{HZ}, $S$ satisfies the PDE%
\begin{equation}
\frac{\partial S}{\partial t}=\varepsilon\left(  \frac{\partial S}%
{\partial\varepsilon}\right)  ^{2} \label{PDEadditive}%
\end{equation}
with the initial condition%
\begin{equation}
S(0,\lambda,\varepsilon)=\mathrm{tr}[\log(\left\vert x-\lambda\right\vert
^{2}+\varepsilon)]. \label{PDEinitial}%
\end{equation}
Note that no derivatives with respect to $\lambda$ appear, so we really have a
PDE\ in $\varepsilon$ and $t,$ with $\lambda$ entering as a parameter in the
initial conditions.

The PDE (\ref{PDEadditive}) is a first-order, nonlinear PDE of
Hamilton--Jacobi type. We now briefly recap the method of characteristics as
it applies to this equation. See Section 3.3 of the book \cite{Evans} of Evans
and Section 5.1 of \cite{DHK} for more information. We introduce a
\textquotedblleft Hamiltonian\textquotedblright\ function by replacing
$\partial S/\partial\varepsilon$ on the right-hand side of (\ref{PDEadditive})
with a \textquotedblleft momentum\textquotedblright\ variable $p_{\varepsilon
},$ with an overall minus sign:
\[
H(\varepsilon,p_{\varepsilon})=-\varepsilon p_{\varepsilon}^{2}.
\]
We then consider Hamilton's equations for this Hamiltonian, meaning that we
look for curves $\varepsilon(t)$ and $p_{\varepsilon}(t)$ satisfying%
\begin{align}
\frac{d\varepsilon}{dt}  &  =\frac{\partial H}{\partial p_{\varepsilon}
}(\varepsilon(t),p_{\varepsilon}(t))=-2\varepsilon(t)p_{\varepsilon}%
(t)\quad\label{Ham1}\\
\frac{dp_{\varepsilon}}{dt}  &  =-\frac{\partial H}{\partial\varepsilon
}(\varepsilon(t),p_{\varepsilon}(t))=p_{\varepsilon}(t)^{2}. \label{Ham2}%
\end{align}
The initial condition $\varepsilon_{0}$ for $\varepsilon(t)$ is an arbitrary
positive number,%
\[
\varepsilon(0)=\varepsilon_{0},
\]
while the initial condition $p_{\varepsilon,0}$ for $p_{\varepsilon}(t)$ is
obtained from the idea that the momentum variable $p_{\varepsilon}$
corresponds to $\partial S/\partial\varepsilon$:%
\begin{equation}
p_{\varepsilon,0}=\frac{\partial S}{\partial\varepsilon}(0,\lambda
,\varepsilon_{0})=\mathrm{tr}[(\left\vert x-\lambda\right\vert ^{2}%
+\varepsilon_{0})^{-1}]. \label{InitialMomentum}%
\end{equation}

A curve of the form $t\mapsto\varepsilon(t)$, for some choice of
$\varepsilon_{0},$ is called a \textbf{characteristic curve} of the PDE
(\ref{PDEadditive}) with the initial condition (\ref{PDEinitial}). We then
have the first and second Hamilton--Jacobi formulas. These assert that that if
a solution to (\ref{Ham1})--(\ref{Ham2}), with initial momentum given by
(\ref{InitialMomentum}), exists with $\varepsilon(t)>0$ up to some time
$t_{\ast},$ then for all $t<t_{\ast},$ we have%
\begin{align}
S(t,\lambda,\varepsilon(t))  &  =S(0,\lambda,\varepsilon_{0})+tH(\varepsilon
_{0},p_{\varepsilon,0})\label{HJ1}\\
\frac{\partial S}{\partial\varepsilon}(t,\lambda,\varepsilon(t))  &
=p_{\varepsilon}(t). \label{HJ2}%
\end{align}
The initial condition (\ref{InitialMomentum}) ensures that the second
Hamilton--Jacobi formula (\ref{HJ2}) holds at $t=0.$ Since we are interested
in $\partial S/\partial\varepsilon,$ the second Hamilton--Jacobi formula will
be more useful to us than the first.

Now, we can solve (\ref{Ham2}) as a separable equation, then plug the result
into (\ref{Ham1}). Then (\ref{Ham1}) becomes separable as well, and we obtain
the explicit formulas%
\begin{align}
\varepsilon(t)  &  =\varepsilon_{0}(1-tp_{\varepsilon,0})^{2}%
\label{epsilonOfT}\\
p_{\varepsilon}(t)  &  =\frac{p_{\varepsilon,0}}{1-tp_{\varepsilon,0}}.
\label{pOfT}%
\end{align}
Once (\ref{epsilonOfT}) is established, (\ref{pOfT}) is equivalent to the
statement that%
\begin{equation}
\sqrt{\varepsilon(t)}~p_{\varepsilon}(t)=\sqrt{\varepsilon_{0}}p_{\varepsilon
,0}. \label{COM}%
\end{equation}

Note that when $t$ approaches the time
\begin{equation}
t_{\ast}(\lambda,\varepsilon_{0})=\frac{1}{p_{\varepsilon,0}}=\frac
{1}{\mathrm{tr}[(\left\vert x-\lambda\right\vert ^{2}+\varepsilon_{0})^{-1}]},
\label{tstarDef}%
\end{equation}
the solution of the system will cease to exist, because $p_{\varepsilon}(t)$
will approach infinity. We call $t_{\ast}(\lambda,\varepsilon_{0})$ the
\textbf{lifetime} of the solution (\ref{epsilonOfT})--(\ref{pOfT}) to
Hamilton's equations (\ref{Ham1})--(\ref{Ham2}).

Now, our goal is to understand the behavior of $\partial S/\partial
\varepsilon$ near $\varepsilon=0,$ for a fixed $\lambda,$ using the
Hamilton--Jacobi formulas (\ref{HJ1})--(\ref{HJ2}). We therefore want to see
what choice of initial condition $\varepsilon_{0}$ (where the value of
$p_{\varepsilon,0}$ is determined by $\varepsilon_{0}$ as in
(\ref{InitialMomentum})) will cause $\varepsilon(t)$ to be close to zero. Now,
if we simply let $\varepsilon_{0}$ approach zero in (\ref{epsilonOfT}), then
$\varepsilon(t)$ will also approach zero---\textit{provided} that the lifetime
$t_{\ast}(\lambda,\varepsilon_{0})$ is at least $t$ in the limit as
$\varepsilon_{0}$ tends to zero. If on the other hand, the $\varepsilon
_{0}\rightarrow0$ limit of $t_{\ast}(\lambda,\varepsilon_{0})$ is less than
$t,$ it does not make sense to apply the Hamilton--Jacobi formula at time $t$
with $\varepsilon_{0}$ close to 0.

The preceding discussion leads us to consider the limit of $t_{\ast}%
(\lambda,\varepsilon_{0})$ as $\varepsilon_{0}\rightarrow0,$ as follows:%
\begin{equation}
T(\lambda)=\lim_{\varepsilon_{0}\rightarrow0^{+}}t_{\ast}(\lambda
,\varepsilon_{0})=\frac{1}{\mathrm{tr}[\left\vert x-\lambda\right\vert ^{-2}%
]}, \label{Tdef}%
\end{equation}
where $\mathrm{tr}[\left\vert x-\lambda\right\vert ^{-2}]$ is interpreted as
in (\ref{limdSdEpsilon2}). Since $x$ is self-adjoint, we can also write%
\begin{equation}
\mathrm{tr}[\left\vert x-\lambda\right\vert ^{-2}]=\int_{\mathbb{R}}\frac
{1}{\left\vert \xi-\lambda\right\vert ^{2}}~d\mu_{x}(\xi). \label{Neg2Moment}%
\end{equation}
The quantity $\mathrm{tr}[\left\vert x-\lambda\right\vert ^{-2}]$ cannot be
zero but will be infinite for certain values of $\lambda.$ Thus, $T$ cannot be
infinite but is zero when $\mathrm{tr}[\left\vert x-\lambda\right\vert ^{-2}]$
is infinite. We then introduce a domain $\Sigma_{t}$ as%
\begin{equation}
\Sigma_{t}=\left\{  \left.  \lambda\in\mathbb{C}\right\vert T(\lambda
)<t\right\}  . \label{SigmaDef}%
\end{equation}
We anticipate that the strategy of letting $\varepsilon_{0}\rightarrow0$ will
work outside the closure of $\Sigma_{t}.$

We now quote three technical results that we will need; their proof is given
at the end of this subsection.

\begin{lemma}
\label{upperSemi.lem}The function $T$ is upper semicontinuous on $\mathbb{C}$
and therefore the set $\Sigma_{t}$ is open.
\end{lemma}

Recall that a real-valued function $f$ on a metric space is said to be upper
semicontinuous if for all $x,$%
\[
\limsup_{y\rightarrow x}f(y)\leq f(x).
\]

\begin{lemma}
\label{specContain.lem}If $x$ is self-adjoint, then for all $t>0,$ the
spectrum of $x$ is contained in $\overline{\Sigma}_{t}.$
\end{lemma}

\begin{lemma}
\label{Tgreater.lem}If $x$ is self-adjoint, then for all $t>0$ and $\lambda$
outside of $\overline{\Sigma}_{t},$ we have $T(\lambda)>t.$
\end{lemma}

Ho and Zhong show that for $\lambda$ outside $\overline{\Sigma}_{t},$ we can
let $\varepsilon_{0}\rightarrow0^{+}$ in (\ref{HJ1}), with the result that
$\varepsilon(t)\rightarrow0$ as well, giving%
\[
\lim_{\varepsilon\rightarrow0^{+}}S(t,\lambda,\varepsilon)=S(0,\lambda
,0)=\mathrm{tr}[\log(\left\vert x-\lambda\right\vert ^{2})],
\]
where $\mathrm{tr}[\log(\left\vert x-\lambda\right\vert ^{2})]$ is well
defined and harmonic for $\lambda$ outside $\bar{\Sigma}_{t},$ by Lemma
\ref{specContain.lem}. Thus, the Brown measure is zero outside $\overline
{\Sigma}_{t}.$ There is a different analysis in \cite[Section 3.2.2]{HZ} to
actually compute the Brown measure, inside $\Sigma_{t}$, but this does not
concern us here---except for the result \cite[Theorem 3.13]{HZ} that the
support of $\mathrm{Br}_{x+c_{t}}$ is equal to (not just contained in)
$\overline{\Sigma}_{t}.$

We now refine the preceding analysis to show that points $\lambda$ outside
$\overline{\Sigma}_{t}$ are outside the spectrum of $x+c_{t}.$

\begin{theorem}
\label{saPlusCirc.thm}Let $c_{t}$ be circular of variance $t$, let $x$ be
self-adjoint and freely independent of $c_{t},$ and consider the function $S$
in (\ref{Sadditive}). Then for each $\lambda$ outside of $\overline{\Sigma
}_{t},$ the function $\frac{\partial S}{\partial\varepsilon}(t,\lambda
,\varepsilon)$ is analytic at $\varepsilon=0.$ Thus, by Theorem \ref{main.thm}%
, the spectrum of $x+c_{t}$ is contained in $\overline{\Sigma}_{t}.$ Since
\cite[Theorem 3.13]{HZ} tells us that the support of $\mathrm{Br}_{x+c_{t}}$
is exactly $\overline{\Sigma}_{t}$ (and since the Brown measure of any element
is always supported on its spectrum), we conclude that the spectrum of
$x+c_{t}$ coincides with its Brown support.
\end{theorem}

\begin{proof}
We apply the second Hamilton--Jacobi formula (\ref{HJ2}) and the formula
(\ref{pOfT})\ for $p_{\varepsilon}(t)$ to get%
\begin{equation}
\frac{\partial S}{\partial\varepsilon}(t,\lambda,\varepsilon(t))=\frac
{p_{\varepsilon,0}}{1-tp_{\varepsilon,0}}, \label{HJexplicit}%
\end{equation}
where $p_{\varepsilon,0}$ is computed as a function of $\lambda$ and
$\varepsilon_{0}$ by (\ref{InitialMomentum}). We now fix some $\lambda$
outside $\overline{\Sigma}_{t}.$ We will first show that the right-hand side
of (\ref{HJexplicit}) makes sense even when $\varepsilon_{0}$ is slightly
negative. Then we will invert the relationship between $\varepsilon_{0}$ and
$\varepsilon=\varepsilon(t)$ near $\varepsilon_{0}=\varepsilon=0$ and plug the
result into (\ref{HJexplicit}) to obtain the desired analytic extension of
$\partial S/\partial\varepsilon.$

We now fix some $\lambda$ outside $\overline{\Sigma}_{t}.$ By Lemma
\ref{specContain.lem}, $\lambda$ is outside the spectrum of $x$ and therefore
$\left\vert x-\lambda\right\vert ^{2}$ is invertible. In that case,%
\[
p_{\varepsilon,0}=\mathrm{tr}[(\left\vert x-\lambda\right\vert ^{2}%
+\varepsilon_{0})^{-1}]
\]
is actually well defined even when $\varepsilon_{0}$ is slightly negative.
Thus, the map $\varepsilon_{0}\mapsto\varepsilon(t)$ is well defined and
analytic in a neighborhood of $\varepsilon_{0}=0.$ Let us use the notation
\[
\tilde{p}_{\varepsilon,0}=\lim_{\varepsilon_{0}\rightarrow0^{+}}%
p_{\varepsilon,0}=\mathrm{tr}[\left\vert x-\lambda\right\vert ^{-2}]=\frac
{1}{T(\lambda)}.
\]
Now, $T(\lambda)>t$ by Lemma \ref{Tgreater.lem}, which means that $\tilde
{p}_{\varepsilon,0}<1/t,$ so that $1-t\tilde{p}_{\varepsilon,0}>0.$ Thus,
$\varepsilon_{0}\mapsto p_{\varepsilon}(t)$ is also well defined beyond
$\varepsilon_{0}=0.$ Also,
\begin{align*}
\left.  \frac{\partial\varepsilon(t)}{\partial\varepsilon_{0}}\right\vert
_{\varepsilon_{0}=0}  &  =\left.  \left[  (1-tp_{\varepsilon,0})^{2}%
+2\varepsilon_{0}(1-tp_{\varepsilon,0})\right]  \right\vert _{\varepsilon
_{0}=0}\\
&  =(1-t\tilde{p}_{\varepsilon,0})^{2}\\
&  >0.
\end{align*}

Thus, by the inverse function theorem, the map $\varepsilon_{0}\mapsto
\varepsilon(t)$ has an analytic inverse map $E_{t}$ defined near 0. We may
therefore construct an analytic function $f$ defined on $(-\delta,\delta)$ by%
\[
f(\varepsilon)=\left.  \frac{p_{\varepsilon,0}}{1-tp_{\varepsilon,0}%
}\right\vert _{\varepsilon_{0}=E_{t}(\varepsilon)}.
\]
By (\ref{HJexplicit}), this function agrees with $\frac{\partial S}%
{\partial\varepsilon}(t,\lambda,\varepsilon)$ for $\varepsilon\in(0,\delta),$
so that $f$ gives the desired analytic extension.
\end{proof}

We now supply the proof of Lemmas \ref{upperSemi.lem}, \ref{specContain.lem}
and \ref{Tgreater.lem}.

\begin{proof}
[Proof of Lemma \ref{upperSemi.lem}]The function $t_{\ast}(\lambda
,\varepsilon_{0})$ in (\ref{tstarDef}) is continuous in $\lambda$ for
$\varepsilon_{0}>0$. As $\varepsilon_{0}$ decreases to 0, $t_{\ast}%
(\lambda,\varepsilon_{0})$ decreases to $T(\lambda).$ It then follows from an
elementary result (e.g., \cite[Theorem 15.84]{Yeh}) that $T$ is upper
semicontinuous and therefore that $\Sigma_{t}$ is an open set.
\end{proof}

\begin{proof}
[Proof of Lemma \ref{specContain.lem}]Since $x$ is self-adjoint and therefore
normal, we can apply Proposition \ref{normalSupport.prop} to conclude that the
spectrum of $x$ coincides with the support of the law $\mu_{x}$ of $x.$ It
then follows from the first paragraph of the proof of Theorem 3.8 in \cite{HZ}
that $\mathrm{supp}(\mu_{x})$ is contained in $\overline{\Sigma}_{t}$ for
every $t.$ We can give different proof of this last statement as follows.
Since $x$ is self-adjoint, $\mathrm{tr}[\left\vert x-\lambda\right\vert
^{-2}]$ can be computed as in (\ref{Neg2Moment}). (Compare how this quantity
would be computed for general $x$ in (\ref{limdSdEpsilon2}).)\ Then by Lemma
4.5 in \cite{Zhong}, the right-hand side of (\ref{Neg2Moment}) is infinite for
$\mu_{x}$-almost-every $\lambda\in\mathbb{C}.$ It follows from the definition
(\ref{Tdef}) of $T$ that $T(\lambda)=0$ for $\mu_{x}$-almost-every $\lambda.$
Thus, by the definition (\ref{SigmaDef}) of $\Sigma_{t}$, we have that
$\mu_{x}$-almost-every $\lambda$ is in $\Sigma_{t}.$ That is, $\Sigma_{t}$ is
a set of full measure for $\mu_{x}$ and $\overline{\Sigma}_{t}$ is then a
\textit{closed} set of full measure, which must contain $\mathrm{supp}(\mu
_{x}).$
\end{proof}

\begin{proof}
[Proof of Lemma \ref{Tgreater.lem}]We will show that for all $\lambda$ outside
the spectrum of $x,$ we have%
\begin{equation}
\Delta\left(  \frac{1}{T(\lambda)}\right)  =4\mathrm{tr}[\left\vert
(x-\lambda)^{2}\right\vert ^{-2}]>0. \label{LaplaceForm}%
\end{equation}
Then, in light of Lemma \ref{specContain.lem}, (\ref{LaplaceForm}) will hold
for all $\lambda$ outside $\overline{\Sigma}_{t}.$ We will conclude that,
outside $\overline{\Sigma}_{t},$ the function $1/T$ cannot have a weak local
maximum and the function $T$ cannot have a weak local minimum. Now, for
$\lambda$ outside $\overline{\Sigma}_{t},$ we certainly have $T(\lambda)\geq
t.$ If $T(\lambda)$ were equal to $t,$ then at all nearby points
$\lambda^{\prime},$ we would have $T(\lambda^{\prime})\geq t,$ or else
$\lambda$ would be in the closure of $\Sigma_{t}.$ But then $\lambda$ would be
a weak local minimum for $T,$ which we have shown to be impossible.

We now verify (\ref{LaplaceForm}). Fix $\lambda$ outside the spectrum of $x$
and use (\ref{Tdef}) to write%
\[
\frac{1}{T(\lambda)}=\mathrm{tr}[\left\vert x-\lambda\right\vert
^{-2}]=\mathrm{tr}[(x-\lambda)^{-1}(x^{\ast}-\bar{\lambda})^{-1}].
\]
Now, by the standard formula for the derivative of the inverse (e.g.,
\cite[Equation (24)]{PDE}), we have%
\begin{align*}
\frac{\partial}{\partial\lambda}(x-\lambda)^{-1}  &  =-(x-\lambda)^{-1}\left(
\frac{\partial}{\partial\lambda}(x-\lambda)\right)  (x-\lambda)^{-1}%
=(x-\lambda)^{-2}\\
\frac{\partial}{\partial\bar{\lambda}}(x-\lambda)^{-1}  &  =-(x-\lambda
)^{-1}\left(  \frac{\partial}{\partial\bar{\lambda}}(x-\lambda)\right)
(x-\lambda)^{-1}=0,
\end{align*}
with similar formulas for the derivative of $(x^{\ast}-\bar{\lambda})^{-1}.$
Thus, differentiating under the trace, we get%
\begin{align*}
\Delta\mathrm{tr}[(x-\lambda)^{-1}(x^{\ast}-\bar{\lambda})^{-1}]  &
=4\frac{\partial^{2}}{\partial\bar{\lambda}\partial\lambda}\mathrm{tr}%
[(x-\lambda)^{-1}(x^{\ast}-\bar{\lambda})^{-1}]\\
&  =\mathrm{tr}[(x-\lambda)^{-2}(x^{\ast}-\bar{\lambda})^{-2}],
\end{align*}
as claimed.
\end{proof}

\subsection{Arbitrary plus circular\label{arbPlusCirc.sec}}

We now consider the circular case of Zhong's paper, $x+c_{t}$, where $x$ is
freely independent of $c_{t}$ but otherwise arbitrary. We consider the
function $T$ and the domain $\Sigma_{t}$ as in (\ref{Tdef}) and
(\ref{SigmaDef}), but where we no longer assume that $x$ is self-adjoint.
Lemma \ref{upperSemi.lem} still holds, with the same proof. But in this
generality, our methods \textit{do not} allow us to prove Lemma
\ref{specContain.lem}---that the spectrum of $x$ is inside $\overline{\Sigma
}_{t}$. Thus, the proof of Theorem \ref{saPlusCirc.thm} breaks down at this
point. Indeed, the \textit{conclusion} of Theorem \ref{saPlusCirc.thm} is
false for general $x,$ as Example \ref{specBrownSupport.ex} shows. In the
example, the spectrum of $x$ is, by Example \ref{hl.ex}, a disk. But for small
$t,$ the closed domain $\overline{\Sigma}_{t}$ is an annulus, so that
$\sigma(x)$ is not contained in $\overline{\Sigma}_{t}.$

What we \textit{can} prove is the following.

\begin{theorem}
\label{2cond.thm}Let $c_{t}$ be a circular element of variance $t,$ let $x$ be
another element (not necessarily self-adjoint) that is freely independent of
$c_{t}.$ For all $\lambda\in\mathbb{C},$ if (1) $\lambda$ is outside the
spectrum of $x,$ and (2) $T(\lambda)>t,$ then $\lambda$ is outside the
spectrum of $x+c_{t}.$
\end{theorem}

\begin{proof}
Although the paper \cite{HZ} assumes that the element $x$ is self-adjoint, the
derivation of the PDE (\ref{PDEadditive}) does not use this assumption. We may
therefore attempt to follow the argument in the previous subsection. Now, the
function $T$ in (\ref{tDef}) is continuous outside the spectrum $\sigma(x)$ of
$x.$ Thus, if $\lambda$ is outside $\sigma(x)$ and satisfies $T(\lambda)>t,$
then $T$ cannot be in the closure of the set $\Sigma_{t}=\{T<t\}$. Thus, under
the assumptions of the theorem, the point $\lambda$ is outside $\overline
{\Sigma}_{t},$ is outside $\sigma(x),$ and satisfies $T(\lambda)>t.$ At this
point, the proof of Theorem \ref{saPlusCirc.thm} goes through without change.
\end{proof}

We now investigate how we can apply Theorem \ref{2cond.thm}. In this
investigation, the following result will be useful.

\begin{lemma}
\label{Tg.lem}If $\lambda$ is outside the spectrum of $x$ and outside
$\overline{\Sigma}_{t},$ then $T(\lambda)>t.$
\end{lemma}

\begin{proof}
If we assume that $\lambda$ is outside the spectrum of $x,$ then the proof of
Lemma \ref{Tgreater.lem} (from the self-adjoint case) goes through without change.
\end{proof}

For our first application of Theorem \ref{2cond.thm}, we simply make the
conclusion of Lemma \ref{specContain.lem} (from the case that $x$ is
self-adjoint) an assumption.

\begin{corollary}
\label{specIn.cor}Let $c_{t}$ be a circular element of variance $t,$ let $x$
be another element (not necessarily self-adjoint) that is freely independent
of $c_{t}.$ Assume that, for some fixed $t,$ the spectrum of $x$ is contained
in $\overline{\Sigma}_{t}.$ Then the spectrum of $x+c_{t}$ is contained in
$\overline{\Sigma}_{t}.$

Since \cite[Theorem B]{Zhong} tells us that the support of $\mathrm{Br}%
_{x+c_{t}}$ is exactly $\overline{\Sigma}_{t}$ (and since the Brown measure of
any element is always supported on its spectrum), we conclude that%
\[
\sigma(x+c_{t})=\mathrm{supp}(\mathrm{Br}_{x+c_{t}})=\overline{\Sigma}_{t}.
\]

\end{corollary}

\begin{proof}
If $\sigma(x)\subset\overline{\Sigma}_{t},$ then by Lemma \ref{Tg.lem},
Theorem \ref{2cond.thm} will apply to every point outside $\overline{\Sigma
}_{t}.$
\end{proof}

Since our goal is ultimately to prove that the spectrum and Brown support of
$x+c_{t}$ are equal, it is natural to assume that this condition holds at
$t=0,$ that is, that the spectrum and Brown support of $x$ are equal.

\begin{corollary}
\label{equalSupports.cor}If the spectrum and Brown support of $x$ coincide,
then the spectrum of $x$ is contained in $\overline{\Sigma}_{t}$ and Corollary
\ref{specIn.cor} tells us that
\[
\sigma(x+c_{t})=\mathrm{supp}(\mathrm{Br}_{x+c_{t}})=\overline{\Sigma}_{t}.
\]

\end{corollary}

\begin{proof}
By (\ref{Tdef}) and (\ref{SigmaDef}), we have $\mathrm{tr}[\left\vert
x-\lambda\right\vert ^{-2}]\leq1/t$ for $\lambda$ outside $\overline{\Sigma
}_{t}.$ Then, by the last part of Theorem 4.6 in \cite{Zhong}, the Brown
support of $x$---which by assumption equals the spectrum of $x$---is contained
in $\overline{\Sigma}_{t}.$ Thus, Corollary \ref{specIn.cor} applies.
\end{proof}

Even if the Brown support of $x$ is a proper subset of the spectrum of $x,$ we
will still have that $\sigma(x)$ is inside $\overline{\Sigma}_{t}$ for all
sufficiently large $t.$

\begin{corollary}
\label{largeT.cor}For a fixed $x,$ the condition $\sigma(x)\subset
\overline{\Sigma}_{t}$ holds for all sufficiently large $t$ and thus for all
sufficiently large $t,$ we have
\[
\sigma(x+c_{t})=\mathrm{supp}(\mathrm{Br}_{x+c_{t}})=\overline{\Sigma}_{t}.
\]

\end{corollary}

\begin{proof}
The function $T$ is upper semicontinuous by Lemma \ref{upperSemi.lem}. Thus,
$T$ achieves a maximum $T_{\max}$ on the compact set $\sigma(x),$ by an
elementary property of upper semicontinuous functions. Thus, $\sigma(x)$ is
contained in $\Sigma_{t}$ (and therefore also in $\overline{\Sigma}_{t}$) for
all $t>T_{\max}.$
\end{proof}

\begin{remark}
Although our proof of Theorem \ref{2cond.thm} (following the proof of Theorem
\ref{saPlusCirc.thm}) uses the PDE\ method, we could also alternatively use
results from Zhong's paper \cite{Zhong}, which is based on
subordination-function methods. For example, we may look at Equation (3.13) in
\cite{Zhong}. We identify $w$ there with $\sqrt{\varepsilon_{0}}$ here and
$\varepsilon$ there with $\sqrt{\varepsilon(t)}$ here. Then after rearranging
slightly, the formula in \cite{Zhong} becomes%
\[
\sqrt{\varepsilon(t)}=\sqrt{\varepsilon_{0}}\left(  1-t\mathrm{tr}\left[
(\left\vert x-\lambda\right\vert ^{2}+\varepsilon_{0})^{-1}\right]  \right)
,
\]
which agrees with our formula (\ref{epsilonOfT}) for $\varepsilon(t).$ We then
consider the $\gamma=0$ case of Eq. (3.23) in \cite{Zhong}, in which case, the
quantity $z$ there equals $\lambda.$ This equation then says, in our notation,
that%
\[
\sqrt{\varepsilon(t)}\frac{\partial S}{\partial\varepsilon}(t,\lambda
,\varepsilon(t))=\sqrt{\varepsilon_{0}}\frac{\partial S}{\partial\varepsilon
}(0,\lambda,\varepsilon_{0}),
\]
which is equivalent to the second Hamilton--Jacobi formula (\ref{HJ2}) with
$p_{\varepsilon}(t)$ described by (\ref{pOfT}) or (\ref{COM}).
\end{remark}

\subsection{Arbitrary plus elliptic\label{arbPlusElliptic.sec}}

Zhong \cite{Zhong} considers an element of the form
\[
x+g_{t,\gamma}%
\]
where $g_{t,\gamma}$ is as in Section \ref{addStatements.sec} and $x$ is
freely independent of $g_{t,\gamma}.$ The case $\gamma=0$ corresponds to the
case $x+c_{t}$ discussed in the previous subsection.

Although this is not how Zhong attacks the problem, it is possible to analyze
$x+g_{t,\gamma}$ using a PDE method, by adapting the results of \cite{HallHo2}
to the additive setting. We use the notation
\[
S(t,\gamma,\lambda,\varepsilon)
\]
for the regularized log potential (as in (\ref{sDef})) of the element
$x+g_{t,\gamma}.$ The PDE for $S$ would then be%
\begin{equation}
\frac{\partial S}{\partial\gamma}=-\frac{1}{2}\left(  \frac{\partial
S}{\partial\lambda}\right)  ^{2},\label{Spde1}%
\end{equation}
where $\partial/\partial\gamma$ and $\partial/\partial\lambda$ are the
Cauchy--Riemann operators with respect to the \textit{complex} variables
$\gamma$ and $\lambda.$ If we take $\gamma$ to be a real number $u$ and take
the real part of both sides of (\ref{Spde1}), we obtain a PDE in real-variable
form:
\begin{equation}
\frac{\partial S}{\partial u}=-\operatorname{Re}\left[  \left(  \frac{\partial
S}{\partial\lambda}\right)  ^{2}\right]  =-\frac{1}{4}\left[  \left(
\frac{\partial S}{\partial x}\right)  ^{2}-\left(  \frac{\partial S}{\partial
y}\right)  ^{2}\right]  ,\quad\lambda=x+iy.\label{Spde2}%
\end{equation}
(There is no real loss of generality in assuming $\gamma$ to be real, since we
can multiply the $x+g_{t,\gamma}$ by a constant of absolute value 1 to
eliminate the factor of $e^{i\theta}$ in (\ref{gdef}), at which point $\gamma$
becomes real.) Note that no derivatives with respect to $t$ or $\varepsilon$
appear in the PDEs (\ref{Spde1}) and (\ref{Spde2}). 

\begin{remark}
The PDEs (\ref{Spde1}) and (\ref{Spde2}) also arise in the analysis of the
evolution of roots of polynomials when the polynomials evolve according to the
heat flow, as in \cite{HallHoHeat} and \cite{HHJK}. 
\end{remark}

Although there is a PDE that applies to the case of $x+g_{t,\gamma},$ Zhong
instead uses methods of free probability and subordination functions. We now
state the main result of Zhong about this case.

\begin{theorem}
[Zhong]\label{ZhongPush.thm}Fix $t>0$ and $\gamma\in\mathbb{C}$ satisfying
(\ref{gammaAndT}) and define a map $\Phi_{t,\gamma}:\mathbb{C}\rightarrow
\mathbb{C}$ by%
\begin{equation}
\Phi_{t,\gamma}(\lambda)=\lambda+\gamma G_{x+c_{t}}(\lambda), \label{PhiDef}%
\end{equation}
where $G_{x+c_{t}}$ is the Cauchy transform of the Brown measure of $x+c_{t}$:%
\begin{equation}
G_{x+c_{t}}(\lambda)=\int_{\mathbb{C}}\frac{1}{\lambda-z}~d\mathrm{Br}%
_{x+c_{t}}(z). \label{Gct}%
\end{equation}
Then $\Phi_{t,\gamma}$ is continuous and the Brown measure $\mathrm{Br}%
_{x+g_{t,\gamma}}$ of $x+g_{t,\gamma}$ is the push-forward of $\mathrm{Br}%
_{x+c_{t}}$ under $\Phi_{t,\gamma}$:%
\[
\mathrm{Br}_{x+g_{t,\gamma}}=(\Phi_{t,\gamma})_{\ast}(\mathrm{Br}_{x+c_{t}}).
\]

\end{theorem}

See Theorems C and D in \cite{Zhong}. In the case $x=0,$ we have that
$\mathrm{Br}_{c_{t}}$ is uniform on an a disk, $\mathrm{Br}_{g_{t,\gamma}}$ is
uniform on an ellipse, and the restriction of $\Phi_{t,\gamma}$ to the support
disk of $\mathrm{Br}_{c_{t}}$ is real linear. Theorem \ref{ZhongPush.thm} says
that if we fix the element $x$ and the parameter $t$ but vary the parameter
$\gamma$ starting from $\gamma=0,$ the Brown measure of $x+g_{t,\gamma}$
varies in a nice way---as push-forward under the explicit map given in
(\ref{PhiDef}). This sort of push-forward behavior is sometimes referred to as
the \textbf{model deformation phenomenon}:\ deforming the free probability
model (in a specific way) deforms the Brown measure in computable fashion. The
model deformation phenomenon was actually first observed by Hall and Ho
\cite{HallHo2} in the multiplicative setting; see Section \ref{ubtGamma.sec}.

We now state our first result about the spectrum of $x+g_{t,\gamma}.$

\begin{proposition}
\label{specElliptic.prop}Fix $t>0$ and $\gamma\in\mathbb{C}$ with $\left\vert
\gamma\right\vert \leq t$ and let $\Phi_{t,\gamma}$ be as in (\ref{PhiDef}).
Assume that the spectrum $\sigma(x)$ of $x$ is contained in $\overline{\Sigma
}_{t},$ which will hold if $x$ is normal or, more generally, if $\sigma(x)$
coincides with the Brown support of $x.$ Then for all $\lambda$ outside of
$\overline{\Sigma}_{t},$ the point $\Phi_{t,\gamma}(\lambda)$ is outside the
spectrum of $x+g_{t,\gamma}.$
\end{proposition}

\begin{lemma}
\label{InjPhi.lem}Under the assumptions of the proposition, the map
$\Phi_{t,\gamma}$ is injective on the complement of $\overline{\Sigma}_{t}$
and may be computed on $(\overline{\Sigma}_{t})^{c}$ as%
\begin{equation}
\Phi_{t,\gamma}(\lambda)=\lambda+\gamma G_{x}(\lambda),\quad\lambda
\in(\overline{\Sigma}_{t})^{c}. \label{Phix}%
\end{equation}

\end{lemma}

Observe that (\ref{Phix}) involves $G_{x}(\lambda),$ while (\ref{PhiDef})
involves $G_{x+c_{t}}(\lambda).$ We note that in some cases, $\Phi_{t,\gamma}$
is actually a homeomorphism of the whole complex plane onto itself. This
result holds, for example, in these two cases: (1) when $x$ is self-adjoint
with $\left\vert \gamma\right\vert \leq t$ but $\gamma\neq t,$ and (2) when
$\left\vert \gamma\right\vert <t$ and $x$ is $R$-diagonal. See Corollary 6.9
and Theorem 7.8 in \cite{Zhong}. On the other hand, if $x=0$ and $\gamma=t=1$,
then $\overline{\Sigma}_{t}$ is the closed unit disk and the restriction of
$\Phi_{t,\gamma}$ to this disk is the map $\lambda\mapsto2\operatorname{Re}%
(\lambda),$ which is not injective. But even in this case, $\Phi_{t,\gamma}$
remains injective on $(\overline{\Sigma}_{t})^{c}$; it is the conformal map
$\lambda\mapsto\lambda+1/\lambda$ from the complement of the closed unit disk
to the complement of $[-2,2].$ (Take $\gamma=1$ in Example 1.5 in \cite{Zhong}.)

\begin{proof}
[Proof of Lemma \ref{InjPhi.lem}]We first let $\varepsilon\rightarrow0$ in the
second part of Eq. (3.33) in \cite{Zhong}, which tells us that $p_{\lambda
}^{(0)}(w(\varepsilon))=p_{\lambda}^{c,(t)}(\varepsilon),$ where these
quantities are defined in Notation 3.10 of \cite{Zhong}. Now, for $\lambda$
outside $\overline{\Sigma}_{t},$ the quantity $w(\varepsilon)$ will tend to
zero as $\varepsilon\rightarrow0$ by \cite[Lemma 3.5]{Zhong}. But by our
assumptions, $\lambda\in(\overline{\Sigma}_{t})^{c}$ is also outside the
spectrum of $x.$ Thus, the quantity $p_{\lambda}^{(0)}(w(\varepsilon))$ in
Notation 3.10 in \cite{Zhong} will converge to%
\[
\mathrm{tr}[(\lambda-x)^{-1}]=G_{x}(\lambda).
\]
Meanwhile, by Lemma 5.11 in \cite{Zhong}, the quantity $p_{\lambda}%
^{c,(t)}(\varepsilon)$ in Notation 3.10 tends to $G_{x+c_{t}}(\lambda)$ as
$\varepsilon\rightarrow0.$ Thus, letting $\varepsilon\rightarrow0$ in Eq.
(3.33) gives
\begin{equation}
G_{x+c_{t}}(\lambda)=G_{x}(\lambda),\quad\lambda\in(\overline{\Sigma}_{t}%
)^{c}, \label{GequalsG}%
\end{equation}
and (\ref{Phix}) follows.

To prove the claimed injectivity, we use an argument due to Zhong (personal
communication), which he has kindly allowed us to reproduce here. The argument
generalizes the proof of Lemmas 3 and 4 in \cite{BianeFreeHeat}, but with some
additional steps. Assume that $\Phi_{t,\gamma}(\lambda_{1})=\Phi_{t,\gamma
}(\lambda_{2})$ for $\lambda_{1},\lambda_{2}\in(\overline{\Sigma}_{t})^{c}.$
Then, using (\ref{Phix}), we have
\[
\lambda_{1}-\lambda_{2}=-\gamma(G_{x}(\lambda_{1})-G_{x}(\lambda_{2})),
\]
so that%
\begin{equation}
\left\vert \lambda_{1}-\lambda_{2}\right\vert =\left\vert \gamma\right\vert
\left\vert G_{x}(\lambda_{1})-G_{x}(\lambda_{2})\right\vert . \label{Inj}%
\end{equation}
Now, using (\ref{Gct}) with $x+c_{t}$ replaced by $x,$ we get
\[
G_{x}(\lambda_{1})-G_{x}(\lambda_{2})=\int_{\mathbb{C}}\frac{(\lambda
_{2}-\lambda_{1})}{(\lambda_{1}-z)(\lambda_{2}-z)}~d\mathrm{Br}_{x}(z).
\]

Applying the Cauchy--Schwarz inequality then gives%
\begin{equation}
\left\vert G_{x}(\lambda_{1})-G_{x}(\lambda_{2})\right\vert \leq\left\vert
\lambda_{1}-\lambda_{2}\right\vert \left(  \int_{\mathbb{C}}\frac
{1}{\left\vert \lambda_{1}-z\right\vert ^{2}}~d\mathrm{Br}_{x}\int%
_{\mathbb{C}}\frac{1}{\left\vert \lambda_{2}-z\right\vert ^{2}}~d\mathrm{Br}%
_{x}\right)  ^{1/2}. \label{CSresult}%
\end{equation}
But according to \cite[Theorem 4.6]{Zhong},%
\[
\int_{\mathbb{C}}\frac{1}{\left\vert \lambda-z\right\vert ^{2}}~d\mathrm{Br}%
_{x}\leq\mathrm{tr}\left[  \left\vert x-\lambda\right\vert ^{-2}\right]  .
\]
Furthermore, since we assume $\sigma(x)\subset\overline{\Sigma}_{t},$ the
proof of Lemma \ref{Tgreater.lem} applies, showing that
\[
\mathrm{tr}\left[  \left\vert x-\lambda\right\vert ^{-2}\right]  =\frac
{1}{T(\lambda)}<\frac{1}{t}.
\]
for $\lambda$ outside $\overline{\Sigma}_{t}.$ Thus, (\ref{CSresult}) becomes%
\[
\left\vert G_{x}(\lambda_{1})-G_{x}(\lambda_{2})\right\vert <\left\vert
\lambda_{1}-\lambda_{2}\right\vert \frac{1}{t}%
\]
and (\ref{Inj}) becomes%
\begin{equation}
\left\vert \lambda_{1}-\lambda_{2}\right\vert <\left\vert \gamma\right\vert
~\left\vert \lambda_{1}-\lambda_{2}\right\vert \frac{1}{t}. \label{contradict}%
\end{equation}
Since $\left\vert \gamma\right\vert \leq t,$ (\ref{contradict}) would be a
contradiction unless $\left\vert \lambda_{1}-\lambda_{2}\right\vert =0.$ Thus,
we obtain the claimed injectivity.
\end{proof}

\begin{proof}
[Proof of Proposition \ref{specElliptic.prop}]Fix $t>0$ and $\gamma
\in\mathbb{C}$ with $\left\vert \gamma\right\vert \leq t$. Consider the
regularized log potential of $x+g_{t,\gamma}$:%
\[
S(t,\gamma,\lambda,\varepsilon)=\mathrm{tr}[(\log(\left\vert x+g_{t,\gamma
}-\lambda\right\vert ^{2}+\varepsilon)].
\]
When $\gamma=0,$ we obtain the function $S(t,\lambda,\varepsilon)$ in the
previous subsection. For $\varepsilon>0,$ define a regularized version of
$\Phi_{t,\gamma},$ denoted~$\Phi_{t,\gamma}^{(\varepsilon)}$ by%
\[
\Phi_{t,\gamma}^{(\varepsilon)}(\lambda)=\lambda+\gamma G_{x+c_{t}%
,\varepsilon}(\lambda).
\]
Here, $G_{x+c_{t},\varepsilon}$ is the Cauchy transform of the
\textit{regularized} Brown measure of $x+c_{t},$ obtained by taking the
Laplacian of the regularized log potential $S(t,\lambda,\varepsilon)$ of
$x+c_{t},$ as in (\ref{Sadditive}). That is to say,
\[
G_{x+c_{t},\varepsilon}(\lambda)=\int_{\mathbb{C}}\frac{1}{\lambda-z}\frac
{1}{4\pi}\Delta_{\lambda}S(t,\lambda,\varepsilon)~d^{2}z.
\]

Define%
\begin{equation}
z(t,\gamma,\lambda,\varepsilon)=\Phi_{t,\gamma}^{(\varepsilon)}(\lambda).
\label{zdef}%
\end{equation}
Then we have
\begin{equation}
\frac{\partial S}{\partial\varepsilon}(t,\gamma,z(t,\gamma,\lambda
,\varepsilon),\varepsilon)=\frac{\partial S}{\partial\varepsilon}%
(t,0,\lambda,\varepsilon). \label{conservePepsilon}%
\end{equation}
This result is the first relation in \cite[Corollary 3.11]{Zhong}, where the
relation between $z$ and $\lambda$ there is given by Equation (3.27) and the
last displayed equation in the proof of Proposition 5.2. One may also obtain
(\ref{conservePepsilon}) from the Hamilton--Jacobi analysis of the PDE in
(\ref{Spde1}) or (\ref{Spde2}). From that perspective, (\ref{conservePepsilon}%
) amounts to second Hamilton--Jacobi formula (\ref{HJ2}), along with the fact
that---since $\varepsilon$ does not appear explicitly on the right-hand side
of (\ref{Spde1}) or (\ref{Spde2})--- $p_{\varepsilon}$ is a constant of motion.

Now, by Corollary \ref{specIn.cor}, points $\lambda$ outside $\overline
{\Sigma}_{t}$ are also outside the spectrum of $x+c_{t},$ from which it
follows that $\left\vert x+c_{t}-\lambda\right\vert ^{2}$ is invertible. Near
any such $\lambda,$ the regularized log potential $S(t,\lambda,\varepsilon)$
of $x+c_{t}$ is defined and analytic in $\varepsilon,$ even for $\varepsilon$
slightly negative. We now define, for each fixed $t$ and $\gamma,$ a map $F$
given by%
\[
F(\lambda,\varepsilon)=(z(t,\gamma,\lambda,\varepsilon),\varepsilon),
\]
where for $\lambda\in(\overline{\Sigma}_{t})^{c},$ we allow $\varepsilon$ to
be slightly negative and where $z(t,\gamma,\lambda,\varepsilon)$ is as in
(\ref{zdef}).

We then consider the matrix of derivatives $F_{\ast}$ of $F$ at $\varepsilon
=0$ with $\lambda$ outside $\overline{\Sigma}_{t},$ which will have the form%
\[
F_{\ast}(\lambda,0)=\left(
\begin{array}
[c]{cc}%
(\Phi_{t,\gamma})_{\ast} & \ast\\
0 & 1
\end{array}
\right)  .
\]
Now, the support of the Brown measure of $x+c_{t}$ is contained in
$\overline{\Sigma}_{t}$ \cite[Theorem 3.8]{HZ}, from which we can see that the
map $\Phi_{t,\gamma}$ in (\ref{PhiDef}) is holomorphic on $(\overline{\Sigma
}_{t})^{c}.$ Thus, $(\Phi_{t,\gamma})_{\ast}$ is just the holomorphic
derivative (a complex number), interpreted as a $2\times2$ matrix. But by
Lemma \ref{InjPhi.lem}, $\Phi_{t,\gamma}$ is injective on $(\overline{\Sigma
}_{t})^{c},$ which means that the holomorphic derivative can never vanish.

We conclude that $F_{\ast}(\lambda,0)$ is invertible for all $\lambda$ outside
$\overline{\Sigma}_{t}.$ It follows that $F$ has a real-analytic inverse
defined near $(\Phi_{t,\gamma}(\lambda),0).$ We can then use the $\lambda
$-component of $F^{-1}(z,\varepsilon)$ to define a function $\Lambda
(z,\varepsilon)$ such that $F(\Lambda(z,\varepsilon),\varepsilon
)=(z,\varepsilon)$ for $(z,\varepsilon)$ in a neighborhood of $(\Phi
_{t,\gamma}(\lambda),0).$ Then we have%
\begin{equation}
\frac{\partial S}{\partial\varepsilon}(t,\gamma,z,\varepsilon)=\frac{\partial
S}{\partial\varepsilon}(t,0,\Lambda(z,\varepsilon),\varepsilon)
\label{SepsInv}%
\end{equation}
for $\varepsilon>0$ and the right-hand side of (\ref{SepsInv}) provides a
real-analytic extension of $\frac{\partial S}{\partial\varepsilon}%
(t,\gamma,z,\varepsilon)$ to $\varepsilon$ in a neighborhood of 0. Therefore,
Theorem \ref{main.thm} applies and $\Phi_{t,\gamma}(\lambda)$ will be outside
the spectrum of $x+g_{t,\gamma}.$
\end{proof}

\begin{theorem}
\label{ellipticEquality.thm}Assume that the spectrum $\sigma(x)$ of $x$ is
contained in $\overline{\Sigma}_{t},$ which will hold if $x$ is normal or,
more generally, if $\sigma(x)$ coincides with the Brown support of $x.$ For
all $\gamma\in\mathbb{C}$ with $\left\vert \gamma\right\vert \leq t,$ define%
\begin{equation}
E_{t,\gamma}=\Phi_{t,\gamma}(\overline{\Sigma}_{t}). \label{EtDef}%
\end{equation}
Then the spectrum and the Brown support of $x+g_{t,\gamma}$ both agree with
$E_{t,\gamma}.$ In particular, the spectrum and the Brown support of
$x+g_{t,\gamma}$ are equal.
\end{theorem}

\begin{lemma}
\label{pushSupport.lem}Suppose $\mu$ is a compactly supported probability
measure on $\mathbb{C}$ and $\Phi:\mathbb{C}\rightarrow\mathbb{C}$ is a
continuous map. Let $\Phi_{\ast}(\mu)$ denote the push-forward of $\mu$ under
$\Phi.$ Then%
\[
\mathrm{supp}(\Phi_{\ast}(\mu))=\Phi(\mathrm{supp}(\mu)).
\]

\end{lemma}

\begin{proof}
In general, a point belongs to the support of a measure if and only if every
neighborhood of the point has positive measure. Note that $\Phi(\mathrm{supp}%
(\mu))$ is compact and therefore closed. Thus, if $z$ is outside
$\Phi(\mathrm{supp}(\mu)),$ some neighborhood $U$ of $z$ is disjoint from
$\Phi(\mathrm{supp}(\mu)).$ Thus, $\Phi^{-1}(U)$ is an open set contained in
the complement of $\mathrm{supp}(\mu),$ so that $\Phi^{-1}(U)$ has measure
zero with respect to $\mu$ and $U$ has measure zero with respect to
$\Phi_{\ast}(U).$ Thus, $z$ is not in $\mathrm{supp}(\Phi_{\ast}(\mu)).$

In the other direction, suppose $z$ is in $\Phi(\mathrm{supp}(\mu)),$ meaning
that $z=\Phi(\lambda)$ for some $\lambda$ in $\mathrm{supp}(\mu).$ Then for
every neighborhood $U$ of $z,$ the set $\Phi^{-1}(U)$ is open and contains
$\lambda,$ so that $\mu(\Phi^{-1}(U))>0$ and, thus, $\Phi_{\ast}(\mu)(U)>0.$
Thus, $z$ is in $\mathrm{supp}(\Phi_{\ast}(\mu)).$
\end{proof}

\begin{lemma}
\label{PhiSurj.lem}The map $\Phi_{t,\gamma}$ is continuous and maps
$\mathbb{C}$ \emph{onto} $\mathbb{C}.$
\end{lemma}

\begin{proof}
The continuity of $\Phi_{t,\gamma}$ follows from Lemma 5.11 in \cite{Zhong}.
We will then follow one of the standard proofs of the fundamental theorem of
algebra, using the concept of the fundamental group (e.g., \cite[Theorem
56.1]{Munkres}). The Cauchy transform $G_{x+c_{t}}(\lambda)$ in the definition
(\ref{PhiDef}) of $\Phi_{t,\gamma}(\lambda)$ behaves like $1/\lambda$ near
infinity. Thus,%
\begin{equation}
\Phi_{t,\gamma}(\lambda)\approx\lambda+\frac{\gamma}{\lambda}\approx
\lambda\label{PhiApprox}%
\end{equation}
near infinity.

Suppose that some $z\in\mathbb{C}$ failed to be in the image of $\Phi
_{t,\gamma}.$ Then $\Phi_{t,\gamma}$ would map $\mathbb{C}$ continuously into
the punctured plane $\mathbb{C}\setminus\{z\}.$ Now, if we restrict
$\Phi_{t,\gamma}$ to a large circle $C$ centered at the origin, then by
(\ref{PhiApprox}), $\Phi_{t,\gamma}(C)$ will have winding number 1 around $z$
and will therefore be homotopically nontrivial in $\mathbb{C}\setminus\{z\}.$
But on the other hand, $\mathbb{C}$ is simply connected, so the image under
$\Phi$ of any loop in $\mathbb{C}$ must be homotopically trivial in
$\mathbb{C}\setminus\{z\}.$ We therefore have a contradiction.
\end{proof}

\begin{proof}
[Proof of Theorem \ref{ellipticEquality.thm}]Zhong's result in Theorem
\ref{ZhongPush.thm} says that $\mathrm{Br}_{x+g_{t,\gamma}}$ is the
push-forward of $\mathrm{Br}_{x+c_{t}}$ under $\Phi_{t,\gamma}.$ Thus, Lemma
\ref{pushSupport.lem} tells us that the support of $\mathrm{Br}_{x+g_{t,\gamma
}}$ is the set $E_{t,\gamma}$ in (\ref{EtDef}). Now, every $z$ outside
$E_{t,\gamma}$ has the form $\Phi_{t,\gamma}(\lambda)$ for some $\lambda
\in\mathbb{C}$, by Lemma \ref{PhiSurj.lem}. But since $z$ is not in
$E_{t,\gamma}=\Phi_{t,\gamma}(\overline{\Sigma}_{t}),$ this $\lambda$ cannot
be in $\overline{\Sigma}_{t}.$ Thus, by Proposition \ref{specElliptic.prop},
$z$ is outside the spectrum of $x+g_{t,\gamma}.$ We conclude that
\[
\sigma(x+g_{t,\gamma})\subset E_{t,\gamma}=\mathrm{supp}(\mathrm{Br}%
_{x+g_{t,\gamma}}).
\]
Since the reverse inclusion $\mathrm{supp}(\mathrm{Br}_{x+g_{t,\gamma}%
})\subset\sigma(x+g_{t,\gamma})$ is a general property of Brown measures, we
have the desired equality.
\end{proof}

\section{Multiplicative case}

\subsection{The case of $ub_{t}$\label{ubt.sec}}

We begin by considering an element of the form $ub_{t},$ where $b_{t}$ is the
free multiplicative Brownian motion defined in Section \ref{mult1.sec} with
$\gamma=0,$ and where $u$ is a unitary element that is freely independent of
$b_{t}.$ We let $\mu_{u}$ denote the law of $u,$ as in (\ref{lawDef}).

We then introduce the regularized log potential of $ub_{t},$ as in
(\ref{sDef}),
\[
S(t,\lambda,\varepsilon)=\mathrm{tr}[\log(\left\vert ub_{t}-\lambda\right\vert
^{2}+\varepsilon)].
\]
According to \cite[Theorem 2.7]{DHK}, $S$ satisfies the PDE%
\begin{equation}
\frac{\partial S}{\partial t}=\varepsilon\frac{\partial S}{\partial
\varepsilon}\left(  1+(\left\vert \lambda\right\vert ^{2}-\varepsilon
)\frac{\partial S}{\partial\varepsilon}-x\frac{\partial S}{\partial x}-y\frac
{\partial S}{\partial y}\right)  ,\quad\lambda=x+iy. \label{multSpde}%
\end{equation}
(Although \cite{DHK} assumes $u=1,$ the derivation of the PDE\ there does not
use this assumption.)

The Hamilton--Jacobi analysis of the PDE (\ref{multSpde}) then proceeds
similarly to the additive case in Sections \ref{saPlusCirc.sec} and
\ref{arbPlusCirc.sec}. One important difference in the two cases is that
$\lambda$ and derivatives of $S$ with respect to the real and imaginary parts
of $\lambda$ now appear on the right-hand side of (\ref{multSpde}). We must
then incorporate $\lambda$ and an associated momentum variable $p_{\lambda}$
into the Hamiltonian system, with the initial value of $p_{\lambda}$ given by%
\begin{equation}
p_{\lambda}(0)=\frac{\partial S}{\partial\lambda}(t,\lambda_{0},\varepsilon
_{0})=-\mathrm{tr}[(ub_{t}-\lambda)^{\ast}((\left\vert ub_{t}-\lambda
\right\vert ^{2}+\varepsilon)^{-1}]. \label{plambdaInitial}%
\end{equation}
Then the second Hamilton--Jacobi formula for $\partial S/\partial\varepsilon$
takes the form:%
\begin{equation}
\frac{\partial S}{\partial\varepsilon}(t,\lambda(t),\varepsilon
(t))=p_{\varepsilon}(t). \label{multHJ2}%
\end{equation}
(Compare (\ref{HJ2}) in the additive case, where $\lambda$ does not depend on
$t.$)

Now, according to Proposition 5.9 in \cite{DHK}, it is possible to solve for
the function $p_{\varepsilon}(t)$ explicitly in terms of the initial
conditions of the system. It is then possible to compute the $\varepsilon
_{0}\rightarrow0$ limit of the lifetime as
\begin{equation}
T(\lambda)=\frac{1}{\tilde{p}_{\varepsilon,0}(\lambda)}\frac{\log(\left\vert
\lambda\right\vert ^{2})}{\left\vert \lambda\right\vert ^{2}-1},
\label{multTdef}%
\end{equation}
where at $\left\vert \lambda\right\vert =1,$ we assign $\log(\left\vert
\lambda\right\vert ^{2})/(\left\vert \lambda\right\vert ^{2}-1)$ its limiting
value, namely 1. Here $\tilde{p}_{\varepsilon,0}$ is the initial value of the
momentum $p_{\varepsilon},$ evaluated at $\varepsilon_{0}=0,$ namely%
\begin{equation}
\tilde{p}_{\varepsilon,0}(\lambda)=\mathrm{tr}\left[  \left\vert
\lambda-u\right\vert ^{-2}\right]  =\int_{S^{1}}\frac{1}{\left\vert
\lambda-\xi\right\vert ^{2}}~d\mu_{u}(\xi). \label{multPepsilon}%
\end{equation}
These calculations do not depend on the assumption that $u=1$ in \cite{DHK}.
We then define%
\[
\Sigma_{t}=\left\{  \left.  \lambda\in\mathbb{C}\right\vert ~T(\lambda
)<t\right\}  .
\]

We now state three technical lemmas, parallel to the ones in Section
\ref{saPlusCirc.sec}, that we will use in the proof of our main result. Their
proofs are given at the end of this subsection.

\begin{lemma}
\label{multUpper.lem}The function $T(\lambda)$ equals $+\infty$ at $\lambda=0$
and is finite elsewhere. Furthermore, the function $T$ is upper semicontinuous
on $\mathbb{C}$ and therefore the set $\Sigma_{t}$ is open.
\end{lemma}

\begin{lemma}
\label{multSpecContain.lem}For all $t>0,$ the spectrum of the unitary element
$u$ is contained in $\overline{\Sigma}_{t}.$
\end{lemma}

\begin{lemma}
\label{multTgreater.lem}For all $t>0$ and all $\lambda$ outside $\overline
{\Sigma}_{t},$ we have $T(\lambda)>t.$
\end{lemma}

We now state the main result of this section.

\begin{theorem}
\label{ubtSpec.thm}For all $t>0,$ the spectrum of $ub_{t}$ is contained in
$\overline{\Sigma}_{t}.$ Thus, since \cite[Theorem 4.28]{HZ} shows that
$\mathrm{supp}(\mathrm{Br}_{ub_{t}})=\overline{\Sigma}_{t}$ and since the
support of the Brown measure of any element is contained in its spectrum, we
conclude that
\[
\sigma(ub_{t})=\mathrm{supp}(\mathrm{Br}_{ub_{t}}).
\]

\end{theorem}

\begin{proof}
Since both $b_{t}$ \cite[p. 265]{BianeJFA} and $u$ are invertible, 0 is not in
the spectrum of $ub_{t}.$

Assume than that $\lambda$ is a nonzero point outside $\overline{\Sigma}_{t}.$
By Lemma \ref{multSpecContain.lem}, $\lambda$ is outside the spectrum of $u.$
In that case, the initial momentum $p_{\varepsilon,0}$ at the point $\lambda$
with $\varepsilon_{0}=0$ (as in (\ref{multPepsilon})) is well defined and
finite, and similarly for the initial momentum $p_{\lambda}$ in
(\ref{plambdaInitial}). Indeed, these initial conditions remain well defined
even when $\varepsilon_{0}$ is slightly negative.

If we evaluate at $\varepsilon_{0}=0,$ the lifetime of the solution of the
Hamiltonian system is $T(\lambda),$ which is greater than $t$ by Lemma
\ref{multTgreater.lem}. Then a general result about flows (e.g., the fact that
the sets $M_{t}$ in \cite[Theorem 9.12]{Lee} are open) tells us that for
$(\lambda_{0},\varepsilon_{0})$ in a neighborhood of $(\lambda,0),$ the
lifetime of the solution with initial conditions $(\lambda_{0},\varepsilon
_{0})$ will remain greater than $t.$\footnote{This point is more subtle than
it may appear because the formula for the lifetime of the Hamiltonian system
in \cite[Proposition 5.9]{DHK} is only valid for $\varepsilon_{0}\geq0.$ The
difficulty is that when $\varepsilon_{0}<0,$ the blowup time of the whole
system may be smaller than the blowup time in the formula for $p_{\varepsilon
}(t)$. See Remark 5.10 in \cite{DHK}.}

We may then consider a map $U_{t}$ given by
\[
U_{t}(\lambda_{0},\varepsilon_{0})=(\lambda(t),\varepsilon(t)),
\]
where the characteristic curves $\lambda(\cdot)$ and $\varepsilon(\cdot)$ are
computed using the initial conditions $\lambda_{0}$ and $\varepsilon_{0}$ and
where the map is defined and analytic in a neighborhood of $(\lambda,0).$ Then
by the proof of Lemma 6.3 in \cite{DHK}, the Jacobian of $U_{t}$ at
$(\lambda_{0},0)$ is invertible. Thus, $U_{t}$ has a real-analytic inverse
defined near $(\lambda_{0},0).$

We then apply the second Hamilton--Jacobi formula
\[
\frac{\partial S}{\partial\varepsilon}(t,\lambda(t),\varepsilon
(t))=p_{\varepsilon}(t;\lambda_{0},\varepsilon_{0})
\]
from \cite[Equation (5.8)]{DHK}, where the notation means that $p_{\varepsilon
}$ is computed using the initial conditions $\lambda_{0}$ and $\varepsilon
_{0}.$ Then%
\[
p_{\varepsilon}(t;U_{t}^{-1}(\lambda,\varepsilon))
\]
will be analytic in $\varepsilon$ in a neighborhood of $\varepsilon=0$ and
will agree with $\frac{\partial S}{\partial\varepsilon}(t,\lambda
,\varepsilon)$ when $\varepsilon>0.$ Thus, Theorem \ref{main.thm} applies and
$\lambda$ is outside the spectrum of $ub_{t}.$
\end{proof}

We conclude the argument by supplying the proofs of Lemmas \ref{multUpper.lem}%
, \ref{multSpecContain.lem}, and \ref{multTgreater.lem}.

\begin{proof}
[Proof of Lemma \ref{multUpper.lem}]As we have noted, $T(\lambda)$ has a
removable singularity at $\left\vert \lambda\right\vert =1.$ Furthermore,
$p_{\varepsilon,0}(\lambda)$ can be infinity but cannot be zero. Thus, the
only way $T(\lambda)$ can be infinite is when $\lambda=0.$ Meanwhile,
according to Proposition 4.8 in \cite{HZ}, the function $T(\lambda)$ is the
decreasing limit as $\varepsilon_{0}\rightarrow0^{+}$ of a certain function
$t_{\ast}(\lambda,\varepsilon_{0})$ and this function is continuous in
$\lambda$ for each $\varepsilon_{0}.$ The claimed semicontinuity of $T$ then
follows by an elementary result \cite[Theorem 15.84]{Yeh} about semicontinuous functions.
\end{proof}

\begin{proof}
[Proof of Lemma \ref{multSpecContain.lem}]Since $u$ is unitary, Proposition
\ref{normalSupport.prop} tells us that the spectrum of $u$ equals the support
of the law $\mu_{u}$ of $u.$ Now, by Lemma 4.5 in \cite{Zhong}, the quantity
$p_{\varepsilon,0}(\lambda)$ in (\ref{multPepsilon}) is infinite for $\mu_{u}%
$-almost every $\lambda.$ Thus, $T(\lambda)=0$ for $\mu_{u}$-almost every,
showing that $\Sigma_{t}$ is a set of full measure for $\mu_{u}.$ Thus,
$\overline{\Sigma}_{t}$ is a closed set of full measure for $\mu_{u},$ showing
that $\sigma(u)=\mathrm{supp}(\mu_{u})$ is contained in $\overline{\Sigma}%
_{t}.$
\end{proof}

\begin{proof}
[Proof of Lemma \ref{multTgreater.lem}]The claimed result is stated in Theorem
4.10 of \cite{HZ}. There is, however, a small gap in the proof, concerning the
case $\left\vert \lambda\right\vert =1,$ which we fill in here. Since
$T(0)=\infty,$ we only consider nonzero $\lambda.$ Now, a point $\lambda$ with
$T(\lambda)=t>0$ will be outside $\overline{\Sigma}_{t}$ if and only if $T$
has a weak local minimum at $\lambda$ (meaning that all $\lambda^{\prime}$ in
a neighborhood of $\lambda$ have $T(\lambda^{\prime})\geq t,$ so that such
points are outside $\Sigma_{t}$). According to \cite[Lemma 4.15]{HZ}, the
function $T(re^{i\theta}),$ with $\theta$ fixed, is strictly increasing for
$1<r<\infty,$ and strictly decreasing for $0<r<1.$ Thus, any possible weak
local minimum of $T$ would have to be at a point on the unit circle.

Consider, then, a point $e^{i\theta}$ on the unit circle that is outside
$\overline{\Sigma}_{t}$ and thus (Lemma \ref{multSpecContain.lem}) outside
$\sigma(u)=\mathrm{supp}(\mu_{u}).$ Putting $\left\vert \lambda\right\vert =1$
in (\ref{multTdef}), we may follow the proof of Proposition 3.5 in \cite{ZhongUnitary} and compute that
\begin{align}
\frac{d^{2}}{d\theta^{2}}\frac{1}{T(e^{i\theta})}  &  =\frac{d^{2}}%
{d\theta^{2}}\int_{S^{1}}\frac{1}{2(1-\cos(\theta-\phi))}~d\mu_{u}(e^{i\theta
})\nonumber\\
&  =\frac{1}{2}\int_{S^{1}}\frac{(2+\cos(\theta-\phi))}{(1-\cos(\theta
-\phi))^{2}}~d\mu_{u}(e^{i\theta})\nonumber\\
&  >0. \label{1overT}%
\end{align}
Thus, $1/T$ cannot have a weak local maximum at $e^{i\theta}$ and $T$ cannot
have a weak local minimum at $e^{i\theta}.$ (Note, however, that $T$
\textit{can} have a weak local minimum on the unit circle, namely when it is
zero---in which case, (\ref{1overT}) becomes meaningless---but this cannot
happen at points in the unit circle outside $\overline{\Sigma}_{t}.$)
\end{proof}

\subsection{The case of $ub_{t,\gamma}$\label{ubtGamma.sec}}

We consider the general free multiplicative Brownian motion $b_{t,\gamma}$ as
defined in Section \ref{mult1.sec}. We then consider the regularized log
potential of $ub_{t,\gamma}$, as in (\ref{sDef}),%
\[
S(t,\gamma,\lambda,\varepsilon)=\mathrm{tr}[\log((ub_{t,\gamma}-\lambda
)^{\ast}(ub_{t,\gamma}-\lambda)+\varepsilon)].
\]
We use results of Hall--Ho \cite{HallHo2}, where $s$ in \cite{HallHo2}
corresponds to $t$ here and where $\tau$ in \cite{HallHo2} corresponds to
$t-\gamma$ here. According to Theorem 4.2 of \cite{HallHo2}, the function $S$
satisfies the PDE%
\begin{equation}
\frac{\partial S}{\partial\gamma}=-\frac{1}{8}\left(  1-\left(  1-\frac{1}%
{2}\varepsilon\frac{\partial S}{\partial\varepsilon}-2\lambda\frac{\partial
S}{\partial\lambda}\right)  ^{2}\right)  , \label{multSPDE1}%
\end{equation}
where we have adjusted the PDE\ to the \textquotedblleft$\varepsilon
$\textquotedblright\ regularization used here, rather than the
\textquotedblleft$\varepsilon^{2}$\textquotedblright\ regularization used in
\cite{HallHo2}. Note that unlike the PDE (\ref{Spde1}) in the additive case,
derivatives with respect to $\varepsilon$ appear on the right-hand side of
(\ref{multSPDE1}).

We now introduce the multiplicative version of the push-forward map
$\Phi_{t,\gamma}$ in Section \ref{arbPlusElliptic.sec}. In the multiplicative
setting, it is convenient to use the Herglotz function in place of the Cauchy
transform. If $a$ is an element of a tracial von Neumann algebra, we define
\[
J_{a}(\lambda)=\frac{1}{2}\int_{\mathbb{C}}\frac{\xi+\lambda}{\xi-\lambda
}~d\mathrm{Br}_{a}(\xi),
\]
whenever the integral converges. (The factor of $\frac{1}{2}$ is a convenient
normalization that makes the formulas in the multiplicative case more similar
to the ones in the additive case.) Note that $J_{a}$ is related to the Cauchy
transform $G_{a}$ as%
\begin{equation}
J_{a}(\lambda)=\frac{1}{2}\int_{\mathbb{C}}\frac{\xi-\lambda+2\lambda}%
{\xi-\lambda}~d\mathrm{Br}_{a}(\xi)=\frac{1}{2}-\lambda G_{a}(\lambda).
\label{HergDef}%
\end{equation}

We then define a map $\Psi_{t,\gamma},$ analogous to the map $\Phi_{t,\gamma}$
in the additive case, by
\begin{equation}
\Psi_{t,\gamma}(\lambda)=\lambda\exp\left\{  \gamma J_{ub_{t}}(\lambda
)\right\}  , \label{PsiDef}%
\end{equation}
where $b_{t}$ is the value of $b_{t,\gamma}$ at $\gamma=0.$ By the $\tau=s$
case of \cite{HallHo2}, this map agrees with the one denoted $\Phi_{s,\tau}$
in \cite[Section 8]{HallHo2}. The following result shows that the
\textquotedblleft model deformation phenomenon\textquotedblright\ holds in
this setting. That is to say, as we vary $\gamma$ with with $u$ and $t$ fixed,
the Brown measure of $ub_{t,\gamma}$ changes in a very specific way, namely by
push-forward under the map $\Psi_{t,\gamma}.$

\begin{theorem}
[Ho--Zhong, Hall--Ho]For all $t>0$ and $\gamma\in\mathbb{C}$ with $\left\vert
\gamma\right\vert \leq t,$ the Brown measure of $ub_{t,\gamma}$ is the
push-forward of $\mathrm{Br}_{ub_{t}}$ under the map $\Psi_{t,\gamma}.$
\end{theorem}

This result is due to Ho--Zhong \cite[Corollary 4.30]{HZ} in the case
$\gamma=0$ and to Hall--Ho \cite[Theorem 8.2]{HallHo2} for $\gamma\neq0.$

\begin{proposition}
\label{multSpecPsi.prop}Fix $t>0$ and $\gamma\in\mathbb{C}$ with $\left\vert
\gamma\right\vert \leq t$ and let $\Psi_{t,\gamma}$ be as in (\ref{PsiDef}).
Then for all $\lambda$ outside of $\overline{\Sigma}_{t},$ the point
$\Psi_{t,\gamma}(\lambda)$ is outside the spectrum of $ub_{t,\gamma}.$
\end{proposition}

We now give the multiplicative version of Lemma \ref{InjPhi.lem}.

\begin{lemma}
\label{injPsi.lem} For all $t>0$ and $\gamma\in\mathbb{C}$ with $\left\vert
\gamma\right\vert \leq t,$ the map $\Psi_{t,\gamma}$ is injective on the
complement of $\overline{\Sigma}_{t}$ and is given on this set by%
\begin{equation}
\Psi_{t,\gamma}(\lambda)=\lambda\exp\left\{  \gamma J_{u}(\lambda)\right\}
,\quad\lambda\in(\overline{\Sigma}_{t})^{c}. \label{Psiu}%
\end{equation}
Thus, $\Psi_{t,\gamma}$ coincides on $(\overline{\Sigma}_{t})^{c}$ with the
holomorphic function denoted $f_{\gamma}$ in \cite[Definition 2.2]{HallHo2}.
Furthermore, $\Psi_{t,\gamma}$ is defined and continuous on $\mathbb{C}.$
\end{lemma}

Note that (\ref{PsiDef}) involves $J_{ub_{t}}$ but (\ref{Psiu}) involves
$J_{u}.$

\begin{proof}
The formula (\ref{Psiu}) follows from the $\tau=s$ case of \cite[Theorem
6.1]{HallHo2}. Once (\ref{Psiu}) is established, we see that $\Psi_{t,\gamma}$
coincides on $(\overline{\Sigma}_{t})^{c}$ with the function denoted
$f_{s-\tau}$ in \cite{HallHo2}, where $s$ and $\tau$ in \cite{HallHo2}
correspond to $t$ and $t-\gamma$, respectively, here. Then the claimed
injectivity follows from Theorem 3.8 in \cite{HallHo2}. Note that this theorem
assumes $\tau\neq0,$ which means $\gamma\neq t$ in our current notation, but
in light of Lemma \ref{multTgreater.lem}, this assumption is only needed to
ensure injectivity of $f_{s-\tau}$ on the boundary of $\Sigma_{t}$;
injectivity on $(\overline{\Sigma}_{t})^{c}$ still holds when $\gamma=t.$

Continuity of $\Psi_{t,\gamma}=f_{\gamma}$ outside $\overline{\Sigma}_{t}$
follows from \cite[Equation (3.8)]{HallHo2}. This formula also allows
computation of the limiting value as we approach the boundary of $\Sigma_{t}.$
Meanwhile, continuity of $\Psi_{t,\gamma}$ on $\overline{\Sigma}_{t}$ follows
from the explicit formulas in \cite[Proposition 8.3]{HallHo2}, which agrees by
construction with the limiting value of $\Psi_{t,\gamma}=f_{\gamma}$ on the boundary.
\end{proof}

\begin{proof}
[Proof of Proposition \ref{multSpecPsi.prop}]The proof follows the proof of
Proposition \ref{specElliptic.prop} in the additive case. We use the
Hamilton--Jacobi analysis for the PDE\ (\ref{multSPDE1}) as developed in
\cite[Section 5]{HallHo2}, but keeping in mind that $\varepsilon$ in
\cite{HallHo2} corresponds to $\sqrt{\varepsilon}$ here. Now, \cite{HallHo2}
uses $\gamma=t$ (i.e., $\tau=0$ in the notation of \cite{HallHo2}) as the
initial condition. But since we have already established Theorem
\ref{ubtSpec.thm} (corresponding to the case $\gamma=0$), it is convenient to
use $\gamma=0$ as our initial condition.

We then have the second Hamilton--Jacobi formula from \cite[Theorem
5.1]{HallHo2} for the PDE (\ref{multSPDE1}), adapted to the notation used
here:%
\[
\frac{\partial S}{\partial\varepsilon}(t,\gamma,\lambda(\gamma),\varepsilon
(\gamma))=p_{\varepsilon}(\gamma),
\]
where the curves $\lambda(\gamma),$ $\varepsilon(\gamma),$ and $p_{\varepsilon
}(\gamma)$ are given explicitly in Eqs. (5.6)--(5.9) of \cite{HallHo2}. Now,
if we take $\gamma=0$ as our initial condition (and use the \textquotedblleft%
$\varepsilon$\textquotedblright\ regularization), then the formulas in
\cite[Equations (5.4) and (5.5)]{HallHo2} for the initial momenta become%
\begin{align}
p_{\lambda}(0)  &  =-\mathrm{tr}\left[  (ub_{t}-\lambda_{0})^{\ast}(\left\vert
ub_{t}-\lambda_{0}\right\vert ^{2}+\varepsilon_{0})\right] \label{plambda0}\\
p_{\varepsilon}(0)  &  =\mathrm{tr}\left[  (\left\vert ub_{t}-\lambda
_{0}\right\vert ^{2}+\varepsilon_{0})\right]  , \label{pepsilon0}%
\end{align}
where $\lambda_{0}$ and $\varepsilon_{0}$ are the initial values of
$\lambda(\gamma)$ and $\varepsilon(\gamma)$, respectively.

We then appeal to Theorem \ref{ubtSpec.thm}, which says that the spectrum of
$ub_{t}$ is contained in $\overline{\Sigma}_{t}.$ Then for $\lambda_{0}%
\in(\overline{\Sigma}_{t})^{c},$ the initial momenta in (\ref{plambda0}) and
(\ref{pepsilon0}) remain well defined and finite even if $\varepsilon_{0}$ is
slightly negative. Thus, Eqs. (5.6)--(5.9) of \cite{HallHo2} make sense and
depend analytically on $\lambda_{0}$ and $\varepsilon_{0},$ even for
$\varepsilon_{0}$ slightly negative.

Now, if we take $\gamma=0$ as our initial condition, the formula for
$\lambda(\gamma)$ at $\varepsilon_{0}=0$ in \cite[Proposition 5.6]{HallHo2}
becomes%
\[
\lambda(\gamma)=f_{\gamma}(\lambda_{0})=\Psi_{t,\gamma}(\lambda_{0}),
\]
where the second equality is from Lemma \ref{injPsi.lem}. Furthermore, at
$\varepsilon_{0}=0,$ we have $\varepsilon(\gamma)=0,$ by \cite[Proposition
5.6]{HallHo2}.

We then define a map $F$, for each fixed $t$ and $\gamma,$ by
\[
F(\lambda_{0},\varepsilon_{0})=(\lambda(\gamma),\varepsilon(\gamma)),
\]
where the curves $\lambda(\cdot)$ and $\varepsilon(\cdot)$ are computed using
the initial conditions $\lambda_{0}$ and $\varepsilon_{0}.$ The Jacobian of
this map at $\varepsilon_{0}=0$ has the form%
\[
F_{\ast}(\lambda_{0},0)=\left(
\begin{array}
[c]{cc}%
(\Psi_{t,\gamma})_{\ast} & \ast\\
0 & \left.  \frac{\partial\varepsilon(\gamma)}{\partial\varepsilon_{0}%
}\right\vert _{\varepsilon_{0}=0}%
\end{array}
\right)  ,
\]
where the quantity in the bottom right corner is easily seen to equal 1, using
the explicit formula for $\varepsilon(\gamma)$ in \cite[Equation
(5.7)]{HallHo2}. The rest of the argument proceeds as in the proof of
Proposition \ref{specElliptic.prop}, using Lemma \ref{injPsi.lem} in place of
Lemma \ref{InjPhi.lem}.
\end{proof}

\begin{theorem}
\label{multEquality.thm}For all $t>0$ and $\gamma\in\mathbb{C}$ with
$\left\vert \gamma\right\vert \leq t,$ the spectrum $\sigma(ub_{t,\gamma})$ of
$ub_{t,\gamma}$ is the image of $\overline{\Sigma}_{t}$ under $\Psi_{t,\gamma
}$ and $\sigma(ub_{t,\gamma})$ coincides with the support of the Brown measure
of $ub_{t,\gamma}.$
\end{theorem}

When $\gamma\neq t,$ the image of $\overline{\Sigma}_{t}$ under $\Psi
_{t,\gamma}$ is the set denoted $\overline{\Sigma}_{s,\tau}$ in
\cite[Definition 2.5]{HallHo2}, where $s$ and $\tau$ in \cite{HallHo2}
correspond to $t$ and $t-\gamma$, respectively, here. See also Figure 4 in
\cite{HallHo2}. When $\gamma=t$ (corresponding to $\tau=0$ in \cite{HallHo2}),
the image of $\overline{\Sigma}_{t}$ under $\Psi_{t,\gamma}$ is the support of
the law of the unitary element $uu_{t},$ where $u_{t}$ is Biane's free unitary
Brownian motion. In the $\gamma=t$ case, the restriction of $\Psi_{t,\gamma}$
to $\overline{\Sigma}_{t}$ is the map written in \cite[Corollary 4.30]{HZ} as%
\[
\lambda\mapsto\Phi_{t,\bar{\mu}}(r_{t}(\theta)e^{i\theta}),
\]
where $r_{t}(\theta)$ is the inner radius of $\overline{\Sigma}_{t}$ at angle
$\theta=\arg\lambda.$

\begin{proof}
[Proof of Theorem \ref{multEquality.thm}]The proof is almost identical to
proofs of Lemma \ref{PhiSurj.lem} and Theorem \ref{ellipticEquality.thm} in
the additive case.
\end{proof}

\subsection{The case of $xb_{t,\gamma}$\label{positive.sec}}

Demni and Hamdi \cite{DemniHamdi} considered an element of the form $pu_{t},$
where $u_{t}$ denotes Biane's free unitary Brownian motion and where $p$ is
nonzero self-adjoint projection, freely independent of $u_{t}.$ They showed
that the support of the Brown measure of $pu_{t}$ is contained in
$\{0\}\cup\overline{\Omega}_{t,\alpha}$ for a certain set $\Omega_{t,\alpha},$
which is bounded by a Jordan curve. Work of Eaknipitsari--Hall \cite{EH}
generalizes this result to elements of the form $xb_{t,\gamma},$ where $x$ is
a non-negative self-adjoint element, assumed to be nonzero and freely
independent of $b_{t,\gamma}.$ The case in which $\gamma=t$ (so that
$b_{t,\gamma}=u_{t}$) and $x$ is a projection corresponds to the results of
\cite{DemniHamdi}.

The PDEs (\ref{multSpde}) and (\ref{multSPDE1}) still hold for the regularized
log potentials of $xb_{t}$ and $xb_{t,\gamma},$ respectively; only the initial
conditions change. We then define two different momenta
\begin{align*}
\tilde{p}_{0}(\lambda)  &  =\mathrm{tr}[\left\vert x-\lambda\right\vert
^{-2}]\\
\tilde{p}_{2}(\lambda)  &  =\mathrm{tr}[\left\vert x\right\vert ^{2}\left\vert
x-\lambda\right\vert ^{-2}],
\end{align*}
where the tilde on $p$ indicates that we are computing with $\varepsilon
_{0}=0.$ Then $\tilde{p}_{0}$ is the initial value of the momentum
$p_{\varepsilon},$ in the limit as $\varepsilon_{0}\rightarrow0$; compare
(\ref{multPepsilon}) in the case of a unitary initial condition. Meanwhile
$\tilde{p}_{2}$ is another similar function that arises in various
computations. We then consider a function $T$ defined as
\[
T(\lambda)=\frac{\log\left(  \frac{\left\vert \lambda\right\vert ^{2}\tilde
{p}_{0}(\lambda)}{\tilde{p}_{2}(\lambda)}\right)  }{\left\vert \lambda
\right\vert ^{2}\tilde{p}_{0}(\lambda)-\tilde{p}_{2}(\lambda)},
\]
with a limiting value of $1/\tilde{p}_{2}(\lambda)$ when $\left\vert
\lambda\right\vert ^{2}\tilde{p}_{0}(\lambda)=\tilde{p}_{2}(\lambda).$ This
function describes the lifetime of Hamiltonian system associated to the PDE
(\ref{multSpde}), with initial condition given by the regularized log
potential of the non-negative element $x,$ in the limit as $\varepsilon
_{0}\rightarrow0.$ See \cite[Definition 3.6 and Proposition 3.7]{EH}.

We note that if $x$ were unitary, then $\left\vert x\right\vert ^{2}$ would
equal $1$ and $\tilde{p}_{2}$ would simply equal $\tilde{p}_{0}.$ In that
case, the formula for $T$ would simplify to
\[
T(\lambda)=\frac{1}{\tilde{p}_{2}(\lambda)}\frac{\log(\left\vert
\lambda\right\vert ^{2})}{\left\vert \lambda\right\vert ^{2}-1},\quad(\text{if
}\left\vert x\right\vert ^{2}=1),
\]
which is just the formula for $T(\lambda)$ in the unitary case. This
observation suggests (correctly!) that the case of a positive initial
condition is much harder to analyze than the case of a unitary initial condition.

The function $T(\lambda)$ is initially defined when $\tilde{p}_{0}(\lambda)$
and $\tilde{p}_{2}(\lambda)$ are finite, which holds outside the spectrum of
$x$ and, thus, outside $[0,\infty)$. By the proof of Proposition 3.10(ii) in
\cite{EH}, if $\tilde{p}_{0}(r)=\infty$ for some $r\in(0,\infty),$ then%
\[
\lim_{\theta\rightarrow0}T(re^{i\theta})=0.
\]
We can, therefore, extend $T(\lambda)$ to be defined for all nonzero complex
numbers $\lambda,$ with $T(\lambda)=0$ whenever $\tilde{p}_{0}(\lambda
)=\infty.$

We then define a set $\Sigma_{t},$ similarly to the previous examples in this
paper, as%
\[
\Sigma_{t}=\left\{  \left.  \lambda\neq0\in\mathbb{C}\right\vert
T(\lambda)<t\right\}  .
\]
Note that, by definition, $0$ is not in $\Sigma_{t};$ the origin will always
be analyzed as a special case.

We now record what is known about the Brown support of $xb_{t}$ and
$xb_{t,\gamma}.$

\begin{itemize}
\item Theorem 3.21 in \cite{EH} asserts that the support of the Brown measure
of $xb_{t}$ is contained in $\{0\}\cup\overline{\Sigma}_{t}.$ Furthermore, if
$0$ is outside $\overline{\Sigma}_{t},$ then $0$ is in the support of the
Brown measure of $xb_{t}$ if and only if $\mu_{x}(\{0\})>0.$ On the other
hand, since \cite{EH} does not compute the Brown measure itself, it is not
known whether the Brown support of $xb_{t}$ fills up all of $\overline{\Sigma
}_{t}.$

\item Proposition 4.12 in \cite{EH} asserts that the support of the Brown
measure of $xb_{t,\gamma}$ is contained in $\{0\}\cup D_{t,\gamma}$ for a
certain closed set $D_{t,\gamma}$ defined in (\ref{DtDef}) below, but it is
not known whether the Brown support fills up all of $D_{t,\gamma}.$
\end{itemize}

We then consider the extent to which Lemmas \ref{multUpper.lem},
\ref{multSpecContain.lem}, and \ref{multTgreater.lem} hold in this setting.

\begin{itemize}
\item Although the function $T$ is not known to be upper semicontinuous, the
set $\Sigma_{t}$ is open \cite[Proposition 3.17]{EH}.

\item According to Corollary 3.13 of \cite{EH}, $\sigma(x)\setminus\{0\}$ is
contained in $\overline{\Sigma}_{t}.$

\item If $\lambda$ is a nonzero complex number outside of $\overline{\Sigma
}_{t}$ and $\lambda$ is outside $(0,\infty),$ then $T(\lambda)>t$. (See the
last part of Proposition 3.19 in \cite{EH}.) However, we cannot exclude the
possibility that there is some $\lambda$ outside $\overline{\Sigma}_{t}$ but
in $(0,\infty)$ with $T(\lambda)=t.$ (Numerically, it appears that such points
$\lambda$ do not exist.)
\end{itemize}

Since the precise Brown support of $xb_{t}$ or $xb_{t,\gamma}$ is not known,
it is not possible to prove that the spectrum and the Brown support coincide.
Furthermore, because of the possibility of points outside $\overline{\Sigma
}_{t}$ with $T(\lambda)=t,$ we cannot exclude the possibility of spectrum
outside $\{0\}\cup\overline{\Sigma}_{t}$ or $\{0\}\cup D_{t,\gamma}.$ We now
state the results we are able to obtain

\begin{theorem}
\label{xbtSpec.thm}If $\lambda$ is a nonzero complex number outside
$\overline{\Sigma}_{t}$ and $T(\lambda)>t,$ then $\lambda$ is outside the
spectrum of $xb_{t}.$ Furthermore, if $0$ is outside $\overline{\Sigma}_{t},$
then $0$ is in the spectrum of $xb_{t}$ if and only if $\mu_{x}(\{0\})>0.$
\end{theorem}

\begin{proof}
Assume $0$ is outside $\overline{\Sigma}_{t}.$ If $\mu_{x}(\{0\})>0,$ then $0$
is in the Brown support of $xb_{t}$ by \cite[Theorem 3.21]{EH} and, therefore,
$0$ is in the spectrum of $xb_{t}.$ On the other hand, if $\mu_{x}(\{0\})=0$
(and $0$ is outside $\overline{\Sigma}_{t}$) then by \cite[Corollary 3.13]%
{EH}, 0 is outside the spectrum of $x.$ Thus, $x$ is invertible. Since, also,
$b_{t}$ is invertible \cite[p. 265]{BianeJFA}, $xb_{t}$ is invertible so that
$0$ is outside the spectrum of $xb_{t}.$

We then consider a nonzero $\lambda$ outside $\overline{\Sigma}_{t}$ and we
\textit{assume} $T(\lambda)>0.$ Then the proof of Theorem \ref{ubtSpec.thm}
applies without change and we conclude that $\lambda$ is outside the spectrum
of $xb_{t}.$
\end{proof}

\begin{figure}[ptb]
\centering
\includegraphics[scale=0.6]{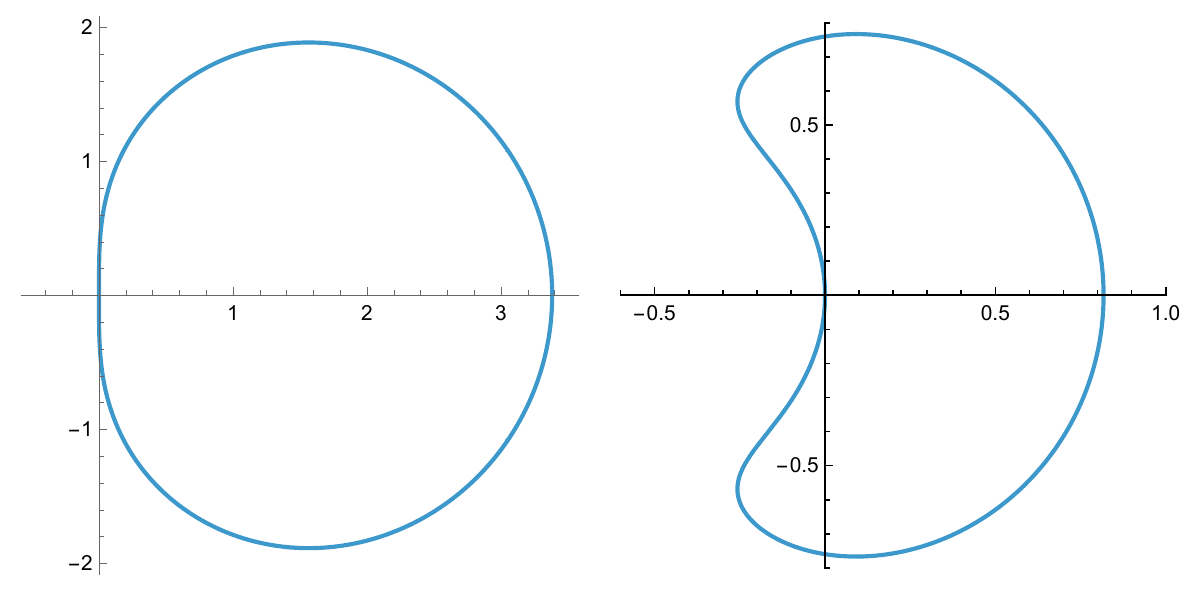}
\caption{The domains $\Sigma_{t}$ (left) and $D_{t,\gamma}$ (right) for
$t=\gamma=2.$}
\label{xbtregions.fig}
\end{figure}

We then define a holomorphic function $f_{\gamma}$ on the
complement of $\overline{\Sigma}_{t}$ by%
\[
f_{\gamma}(\lambda)=\lambda\exp\left\{  \gamma J_{x}(\lambda)\right\}  ,
\]
where $J_{x}$ is the Herglotz transform of $x$ as in (\ref{HergDef}). Outside
$\overline{\Sigma}_{t},$ the function $f_{\gamma}$ plays the role of the map
$\Psi_{t,\gamma}$ from the case of a unitary initial condition. (Compare Lemma
\ref{injPsi.lem}.) According to \cite[Proposition 4.14]{EH}, $f_{\gamma}$ is
injective on $(\overline{\Sigma}_{t})^{c}$ and $f_{\gamma}(\lambda)$ tends to
infinity as $\lambda$ tends to infinity. The closed set $D_{t,\gamma}$ in
\cite[Proposition 4.12]{EH} is then defined by the condition%
\begin{equation}
(D_{t,\gamma})^{c}=f_{\gamma}((\overline{\Sigma}_{t})^{c}). \label{DtDef}%
\end{equation}
See Figure \ref{xbtregions.fig}.

\begin{theorem}
\label{specxbtgamma.thm}Let $z$ be a nonzero complex number outside
$D_{t,\gamma}$ and let $\lambda$ be the complex number outside $\overline
{\Sigma}_{t}$ such that $f_{\gamma}(\lambda)=z.$ If $T(\lambda)>t,$ then $z$
is outside the spectrum of $xb_{t,\gamma}.$ Furthermore, if $0$ is outside
$\overline{\Sigma}_{t},$ then the point $0=f_{\gamma}(0)$ is in the spectrum
of $xb_{t,}\gamma$ if and only if $\mu_{x}(\{0\})>0.$
\end{theorem}

\begin{proof}
The analysis of 0 is similar to the proof of Theorem \ref{xbtSpec.thm}, using
\cite[Proposition 4.13]{EH} in place of \cite[Theorem 3.21]{EH}. If $\lambda$
is a nonzero point outside $\overline{\Sigma}_{t}$ and we assume
$T(\lambda)>0,$ then Theorem \ref{xbtSpec.thm} tells us that $\lambda$ is
outside the spectrum of $xb_{t}.$ Then the proof of Proposition
\ref{multSpecPsi.prop} applies, using the injectivity of $f_{\gamma}$ obtained
in \cite[Proposition 4.14]{EH}, showing that $f_{\gamma}(\lambda)$ is outside
the spectrum of $xb_{t,\gamma}.$
\end{proof}

\subsection*{Acknowledgments:}

We thank Jorge Garza-Vargas and Benson Au for useful discussions on a related
problem that motivated this line of research. We thank Hari Bercovici for
supplying the proof of Lemma \ref{invertibility.lem} and we thank Ping Zhong
for several helpful comments. We also thank the referees for a careful reading of the paper and offering useful suggestions and corrections.

\subsection*{Funding Information:}

Hall's research is supported in part by a grant from the Simons Foundation. Ho's research is supported in part by the NSTC grant 114-2115-M-001-005-MY3.

\end{document}